\documentclass[12pt,a4paper]{amsart}

\usepackage{a4wide}

\usepackage[matrix,arrow,curve,frame]{xy}
\usepackage{amsmath,amsthm,amssymb,enumerate}
\usepackage{latexsym}
\usepackage{amscd}

\usepackage{graphicx}
\usepackage{xcolor}

\usepackage[colorlinks=true]{hyperref}

\newtheorem{theorem}[subsection]{Theorem}
\newtheorem{thm}[subsection]{Theorem}
\newtheorem{defn}[subsection]{Definition}
\newtheorem{prop}[subsection]{Proposition}

\newtheorem{cor}[subsection]{Corollary}
\newtheorem{lemma}[subsection]{Lemma}

\newtheorem{remark}[subsection]{Remark}

\theoremstyle{definition}
\newtheorem{example}[subsection]{Example}

\DeclareMathOperator{\Pp}{P}

\newcommand{\Oo}{{\mathcal O}}

\newcommand{\bm}{{\bf m}}

\newcommand{\R}{{\mathbb R}}
\newcommand{\C}{{\mathbb C}}
\newcommand{\Proj}{{\mathbb P}}
\newcommand{\CP}{{\C\!\Pp}}
\newcommand{\RP}{{\R\!\Pp}}
\newcommand{\Z}{{\mathbb Z}}
\newcommand{\N}{{\mathbb N}}
\newcommand{\T}{{\mathbb T}}

\newcommand{\eps}{\varepsilon}
\newcommand{\pP}{\partial_P}
\newcommand{\pPP}{\partial_{P\times P}}
\newcommand{\pPp}{\partial_{P'}}
\newcommand{\pPd}{\partial_{P''}}

\newcommand{\oz}{\overline{z}}

\newcommand{\Uu}{{\mathcal U}}
\newcommand{\Vv}{{\mathcal V}}

\newcommand{\tM}{\tilde{M}}

\newcommand{\tP}{\tilde{P}}
\newcommand{\tmu}{\tilde{\mu}}

\newcommand{\tR}{\tilde{R}}
\newcommand{\tT}{\tilde{T}}
\newcommand{\tp}{\tilde{p}}
\newcommand{\tq}{\tilde{q}}
\newcommand{\ttau}{\tilde{\tau}}
\newcommand{\tomega}{\tilde{\omega}}
\newcommand{\tpsi}{\tilde{\psi}}
\newcommand{\ttheta}{\tilde{\theta}}
\newcommand{\tvphi}{\tilde{\varphi}}

\DeclareMathOperator{\im}{im}
\DeclareMathOperator{\Ham}{Ham}
\DeclareMathOperator{\interior}{int}

\DeclareMathOperator{\Lie}{Lie}

\begin{document}

\title{Remarks on Lagrangian intersections in toric manifolds} 

\author[M.~Abreu]{Miguel Abreu}
\address{Centro de An\'{a}lise Matem\'{a}tica, Geometria e Sistemas
Din\^{a}micos, Departamento de Matem\'atica, Instituto Superior T\'ecnico, Av.
Rovisco Pais, 1049-001 Lisboa, Portugal}
\email{mabreu@math.ist.utl.pt}

\author[L.~Macarini]{Leonardo Macarini}
\address{Universidade Federal do Rio de Janeiro, Instituto de Matem\'atica,
Cidade Universit\'aria, CEP 21941-909 - Rio de Janeiro - Brazil}
\email{leonardo@impa.br}

\date{\today}

\begin{abstract}
We consider two natural Lagrangian intersection problems in the context of symplectic toric manifolds:
displaceability of torus orbits and of a torus orbit with the real part of the toric manifold. Our remarks
address the fact that one can use simple cartesian product and symplectic reduction considerations 
to go from basic examples to much more sophisticated ones. We show in particular how rigidity results
for the above Lagrangian intersection problems in weighted projective spaces can be combined with
these considerations to prove analogous results for all monotone toric symplectic manifolds. We also
discuss non-monotone and/or non-Fano examples, including some with a continuum of non-displaceable
torus orbits.
\end{abstract}

\thanks{Partially supported by Funda\c c\~ao para a 
Ci\^encia e a Tecnologia (FCT/Portugal), Funda\c c\~ao 
Coordena\c c\~ao de Aperfei\c coamento de Pessoal de N\'{\i}vel Superior 
(CAPES/Brazil) and Conselho Nacional de Desenvolvimento Cient\'{\i}fico 
e Tecnol\'ogico (CNPq/Brazil).}

\maketitle         

\section{Introduction}
\label{sec: intro}

Let $(M^{2n}, \omega)$ be a toric symplectic manifold, i.e. a symplectic manifold equipped with
an effective Hamiltonian $\T^n$-action generated by a moment map
\[
\mu : M \to P := \mu(M) \subset (\R^n)^\ast\,,
\]
where the moment polytope $P$ is defined by
\begin{equation} \label{eq:polytope}
\ell_i (x) := \langle x, \nu_i \rangle + a_i \geq 0\,,\ 
i=1,\ldots,d.
\end{equation}
Here the $a_i$'s are real numbers, each vector $\nu_i \in\Z^n$ is the primitive integral interior
normal to the facet $F_i$ of the polytope $P$ and $d$ is the number of facets of $P$.

Denote by $\tau:M\to M$ the canonical anti-symplectic involution, characterized by
$\mu \circ \tau = \mu$, and let $R := M^\tau$ denote its fixed point set.
$R$ is a Lagrangian manifold, often called the real part of $M$.

Given $x\in\interior (P)$, let $T_x := \mu^{-1} (x)$ denote the corresponding $\T^n$-orbit,
a Lagrangian torus in $M$. Since $\mu \circ \tau = \mu$, we also have that  $\tau (T_x) = T_x$ and 
$(T_x)^{\tau} = T_x \cap R$. This last set, the real part of a regular $\T^n$-orbit, is discrete with 
$2^n$ points.

This can be seen quite explicitly in action-angle coordinates. Consider $\breve{M}\subset M$ 
defined by $\breve{M} = \mu^{-1} (\interior (P))$. One checks that $\breve{M}$ is an open dense subset of $M$, consisting of all the points where the $\T^n$-action is free. It can be described as
\[
\breve{M}\cong \interior (P) \times \T^n = 
\left\{ (x_1,\ldots,x_n, e^{i\theta_1}, \ldots, e^{i\theta_n}): x\in\interior (P)\subset\R^n\,,\ 
\theta\in\R^n/2\pi\Z^n\right\}\,,
\]
where $(x,\theta)$ are symplectic action-angle coordinates for $\omega$, i.e.
$$
\omega |_{\breve{M}} = d x\wedge d \theta = \sum_{j=1}^n d x_j \wedge d \theta_j\ .
$$
In these action-angle coordinates the moment map and anti-symplectic involution are given by
\[
\mu (x,e^{i\theta}) = x \quad\text{and}\quad \tau (x, e^{i\theta}) = (x, e^{-i\theta})\,.
\]
Hence, we have that
\[
R \cap \breve{M} = (x, \pm 1) \equiv (x_1,\ldots,x_n, \pm 1,\ldots,\pm 1)\,,\ x\in\interior(P)\,,
\]
and
\[
\sharp (R\pitchfork T_x) = 2^n\,,\ \forall\, x\in\interior(P)\,.
\]

\begin{example} \label{ex:cpn}
Consider $(\CP^n, \omega_{FS})$ with moment polytope $P\subset (\R^n)^\ast$ given by
\[
P = \left\{ (x_1,\ldots,x_n)\,:\ x_j + 1 \geq 0\,,\ j=1,\ldots,n\,,\ \text{and}\ 
-(\sum_{j=1}^n x_j) + 1 \geq 0 \right\}\,.
\]
Then
\[
R = \RP^n \quad\text{and}\quad T_0 := \mu^{-1} (0) \equiv\  \text{Clifford $n$-torus.}
\]
\end{example}

\begin{example} \label{ex:monotone}
Let $(M^{2n},\omega)$ be a monotone toric symplectic manifold, i.e. such that 
$[\omega] = \lambda (2\pi c_1 (\omega)) \in H^2 (M)$ with $\lambda \in\R^+$.
The corresponding moment polytope $P\subset (\R^n)^\ast$ is then a Fano Delzant
polytope, meaning that it can be defined by~(\ref{eq:polytope}) with
\[
a_1 = \cdots = a_d = \lambda\,.
\]
In this case, the Lagrangian torus fiber over the origin $0\in P$ will be called the 
\emph{centered} or \emph{special} torus fiber. It is the unique monotone torus fiber 
of the monotone toric symplectic manifold $(M^{2n},\omega)$.

Without any loss of generality, we will usually assume that $\lambda = 1$ (as we already 
did in Example~\ref{ex:cpn}).
\end{example}

In this context of toric symplectic manifolds, it is natural to consider the following Lagrangian
intersection rigidity question: given $x\in\interior(P)$, does there exist $\psi \in \Ham (M,\omega)$ such that
\[
\psi (T_x) \cap T_x = \emptyset \quad\text{or}\quad
\psi (T_x) \cap R = \emptyset \quad\text{or}\quad
\psi (R) \cap R = \emptyset \ ?
\]
Our remarks in this paper concern the first two of these Lagrangian intersection problems
and will show how simple cartesian product (Section~\ref{sec:rmk-1}) and symplectic 
reduction (Section~\ref{sec:rmk-2}) considerations can
be used to go from basic examples to more sophisticated ones. 

To illustrate our point, consider the following theorem on the most basic example that can
be considered in this context.

\begin{theorem} \label{thm:basic} For $(\CP^n, \omega_{FS})$, as in Example~\ref{ex:cpn},
we have that:
\begin{itemize}
\item[(i)]~\cite{BEP, Cho1}
\[
\psi (T_0) \cap T_0 \ne \emptyset \quad\text{and}\quad \sharp (\psi (T_0) \pitchfork T_0) \geq 2^n\,,
\ \forall\, \psi \in \Ham (\CP^n,\omega_{FS})\,.
\]
\item[(ii)]~\cite{BC1, BC2, EP, Tamarkin}
\[
\psi (T_0^n) \cap \RP^n \ne \emptyset \,, \ \forall\, \psi \in \Ham (\CP^n,\omega_{FS})\,.
\]
\item[(iii)]~\cite{Alston}
\[
\sharp (\psi (T_0^{2n-1}) \pitchfork \RP^{2n-1}) \geq 2^n \,, \ \forall\, \psi \in \Ham (\CP^{2n-1},\omega_{FS})\,.
\]
\end{itemize}
\end{theorem}

The following results follow by straightforward applications of our remarks.

\begin{theorem} \label{thm:even} (Cf. Remark~\ref{rmk:even} and \S~\ref{ssec:even}.)
\[
\sharp (\psi (T_0^{2n}) \pitchfork \RP^{2n}) \geq 2^n \,, \ \forall\, \psi \in \Ham (\CP^{2n},\omega_{FS})\,.
\]
\end{theorem}

\begin{theorem} \label{thm:genbasic}
Let $(M^{2n}, \omega)$ be a monotone toric symplectic manifold, as in Example~\ref{ex:monotone},
$R$ its real part and $T$ its special centered torus fiber. Let $\nu_1, \ldots, \nu_d \in \Z^n$ denote
the primitive integral interior normals to the facets of the corresponding Fano Delzant polytope
$P\subset(\R^n)^\ast$.
\begin{itemize}
\item[(i)] (Cf. Proposition~\ref{prop:zero-sum} in \S~\ref{ssec:zero-sum}.) If $\sum_{i=1}^d \nu_i = 0$ then
\[
\psi (T) \cap R \ne \emptyset \,,\  \psi (T) \cap T \ne \emptyset 
\quad\text{and}\quad \sharp (\psi (T) \pitchfork T) \geq 2^n
\,, \ \forall\, \psi \in \Ham (M,\omega)\,.
\]
\item[(ii)] (Cf. Proposition~\ref{prop:sym} in \S~\ref{ssec:sym}.) If $P\subset(\R^n)^\ast$ is \emph{symmetric},
i.e. whenever $\nu\in\Z^n$ is the normal to a facet of $P$ then $-\nu\in\Z^n$ is also the normal to a facet of $P$,
then
\[
\sharp (\psi(T) \pitchfork R) \geq 2^n \,, \ \forall\, \psi \in \Ham (M,\omega)\,.
\]
\end{itemize}
\end{theorem}

\begin{remark} \label{rmk:sym} A particular interesting example that fits both (i) and (ii) of this theorem is
$M = \CP^2 \sharp 3 \overline{\CP}^2$, equipped with a monotone symplectic form (cf. \S~\ref{ssec:sym}).
\end{remark}

Consider now the following generalization of part (i) of Theorem~\ref{thm:basic}, contained in the work
of Woodward~\cite{Woodward} and forthcoming work of Cho and Poddar~\cite{Cho3}.

\begin{theorem} \label{thm:weighted}
Let $\CP(1,\bm) := \CP (1, m_1, \ldots, m_n)$, $m_1,\ldots,m_n\in\N$, denote the weighted projective space 
determined by the moment polytope $P _\bm\subset (\R^n)^\ast$ given by
\[
P = \left\{ (x_1,\ldots,x_n)\,:\ x_j + 1 \geq 0\,,\ j=1,\ldots,n\,,\ \text{and}\ 
-(\sum_{j=1}^n m_j x_j) + 1 \geq 0 \right\}
\]
(i.e. the symplectic quotient of $\C^{n+1}$ by the $S^1$-action with weights $(1,m_1, \ldots, m_n)$). 
Let $T_0 := \mu^{-1} (0)$ denote its special centered torus fiber, where $\mu:\CP (1, \bm) \to P_\bm$ is the 
moment map. Then
\[
\psi (T_0) \cap T_0 \ne \emptyset \quad\text{and}\quad \sharp (\psi (T_0) \pitchfork T_0) \geq 2^n\,,
\ \forall\, \psi \in \Ham (\CP(1,\bm),\omega_{\bm})\,.
\]
\end{theorem}

The following result, removing the zero-sum assumption in part (i) of Theorem~\ref{thm:genbasic} and originally 
due to Entov-Polterovich~\cite{EP}, Cho~\cite{Cho2} and Fukaya-Oh-Ohta-Ono~\cite{FOOO2}, follows from this
Theorem~\ref{thm:weighted} by a straightforward application of our symplectic reduction remark 
(cf. Proposition~\ref{prop:monotone} in \S~\ref{ssec:monotone}).

\begin{theorem} \label{thm:monotone}
Let $(M, \omega)$ be a compact monotone toric symplectic manifold, as in Example~\ref{ex:monotone},
and $T$ its special centered torus fiber. Then
\[
\psi (T) \cap T \ne \emptyset \quad\text{and}\quad \sharp (\psi (T) \pitchfork T) \geq 2^n\,,\ \forall\, \psi \in 
\Ham (M,\omega)\,.
\]
\end{theorem}

\begin{remark} \label{rmk:gen-monotone}
An appropriate generalization of part (ii) of Theorem~\ref{thm:basic} to weighted projective spaces would imply, 
by the same straightforward application of our symplectic reduction remark, that on any compact monotone
toric symplectic manifold $(M, \omega)$ we also have that
\[
\psi (T) \cap R \ne \emptyset  \,, \ \forall\, \psi \in \Ham (M,\omega)\,.
\]
This has been proved by Alston-Amorim~\cite{Alston-Amorim} using the methods
developed by Fukaya, Oh, Ohta and Ono in~\cite{FOOO1, FOOO2, FOOO}.
\end{remark}

With the help of another basic example, i.e. the total space of the line bundle $\Oo(-1)\to\CP^1$
(cf. Section~\ref{sec:Fano}), our remarks can also be used to prove interesting non-displaceability results on 
certain non-monotone and/or non-Fano examples, such as:
\begin{itemize}
\item[-] a continuum of non-displaceable torus fibers on $M = \CP^2 \sharp 2 \overline{\CP}^2$ with a certain
non-monotone symplectic form (cf. Example 10.3 in~\cite{FOOO} and \S~\ref{ssec:cont}).
\item[-] a particular non-displaceable torus fiber on the family of non-Fano examples given by 
Hirzebruch surfaces $H_k := \Proj (\Oo (-k) \oplus \C) \to \CP^1$, with $2\leq k \in\N$
(cf. Example 10.1 in~\cite{FOOO} and \S~\ref{ssec:Hirz}).
\item[-] a continuum of non-displaceable torus fibers on a certain non-Fano toric symplectic $4$-manifold considered 
by McDuff (cf. \S~2.1 in~\cite{McDuff2} and \S~\ref{ssec:Dusa}).
\end{itemize}
For all non-displaceable torus fibers in these examples, we can also use our remarks to obtain an appropriate
optimal lower bound for the number of transversal intersection points, which for these Lagrangian $2$-tori is $4$.

\begin{remark} \label{rmk:Dusa}
Regarding this last example, McDuff shows in~\cite{McDuff2} that it gives rise, under a repeated \emph{wedge}
construction, to a monotone symplectic toric $12$-manifold with a continuous interval of Lagrangian torus fibers that
cannot be displaced by her method of probes~\cite{McDuff}. Hence, as she points out, these fibers ``may perhaps
be non-displaceable by Hamiltonian isotopies, even though, according to~\cite{FOOO1}, their Floer homology vanishes".
Although we can use the remarks in this paper to prove non-displaceability of the relevant torus fibers in the non-Fano
toric symplectic $4$-manifold (cf. \S~\ref{ssec:Dusa}), so far we have not been able to do the same for the corresponding
monotone toric symplectic $12$-manifold.
\end{remark}

\subsection*{Acknowledgements} We thank Matthew Strom Borman, Cheol-Hyun Cho, Rui Loja Fernandes, Dusa McDuff 
and Chris Woodward, for several useful discussions regarding this paper.

We thank IMPA and IST for the warm hospitality during the preparation of this work.

Part of these results were first presented by the first author at the Workshop on Symplectic Geometry and Topology,
Kyoto University, Japan, February 14--18, 2011. He thanks the organizers, Kenji Fukaya and Kaoru Ono, for the 
opportunity to participate in such a wonderful event.

%\section{Lagrangian Floer Homology}
%\label{sec:LHF}
%
%This section consists of what Leonardo has written (lfhtoric-10fev10).

\section{First Remark: Cartesian Product}
\label{sec:rmk-1}

This first remark, on cartesian products, is motivated by Alston's result in part (iii) of Theorem~\ref{thm:basic} 
and can be used to remove its dimension restriction, i.e. prove Theorem~\ref{thm:even}. We will also use
it in combination with our second remark, on symplectic reduction.

\subsection{Combinatorial Floer Invariant} \label{sec:CFI}

Let $P\subset(\R^n)^\ast$ be a moment polytope defined by
\[
\ell_i (x) := \langle x, \nu_i \rangle + a_i \geq 0\,,\ 
i=1,\ldots,d.
\]
Each vector $\nu_i = ((\nu_i)_1, \ldots (\nu_i)_n)\in\Z^n$ is the primitive integral normal
to the facet $F_i$ of the polytope $P$, and $d$ is the number of facets of $P$. We will say 
that $P$ is {\bf even} if $d$ is even.

Let $CF^n$ be the vector space of dimension $2^n$ generated over
$\Z_2$ by the following $2^n$ symbols:
\[
(\eps_1,\ldots, \eps_n)\ \text{with}\ \eps_k =\pm 1\,,\ k=1,\ldots,n.
\]
Consider the linear map
\[
\partial_P: CF^n \to CF^n
\]
defined on basis elements by
\[
\pP (\eps) = \sum_{i=1}^d (-1)^{\nu_i} \eps
\]
where $\eps = (\eps_1, \ldots, \eps_n)$ and
\[
(-1)^{\nu_i} \eps = \left((-1)^{(\nu_i)_1} \eps_1, \ldots ,
(-1)^{(\nu_i)_n} \eps_n \right).
\]
\begin{prop}\label{prop:square}
$\pP^2 = 0$ if $d$ is even and $\pP^2 = \rm{id}$ if $d$ is odd.
\end{prop}
\begin{proof}
\begin{align}
\pP(\pP \eps) & = \pP \left(\sum_{i=1}^d (-1)^{\nu_i} \eps \right)
=  \sum_{i=1}^d \pP\left((-1)^{\nu_i} \eps\right) \notag \\
& = \sum_{i,j=1}^d (-1)^{\nu_j}\left((-1)^{\nu_i}\eps \right)
= \sum_{i,j=1}^d (-1)^{\nu_i + \nu_j}\eps \notag \\
& = 2 \left(\sum_{i<j} (-1)^{\nu_i + \nu_j} \eps \right)
+ \sum_{i=1}^d (-1)^{2\nu_i} \eps \notag \\
& = 0 + \sum_{i=1}^d \eps = d \eps \notag
\end{align}
\end{proof}

\begin{defn}\label{defn:HF0}
The Floer invariant $HF(P)$ of an even integral polytope $P$ is defined by
\[
HF(P) := \dim\,\ker(\pP) - \dim\,\im(\pP).
\]
\end{defn}

\begin{prop}\label{prop:kunneth}
If $P = P' \times P''$, with $P'$ and $P''$ even integral polytopes, then
\[
HF(P) = HF (P') \cdot HF (P'')\,.
\]
\end{prop}
\begin{proof}
Suppose $P'\subset\R^{n'}$ and $P''\subset\R^{n''}$ have normal vectors
\[
\nu'_1, \ldots, \nu'_{d'} \in \Z^{n'}
\quad\text{and}\quad
\nu''_1, \ldots, \nu''_{d''} \in \Z^{n''}\,.
\]
Then $P = P' \times P'' \subset \R^{n'} \times \R^{n''} = \R^n$, 
$n = n' + n''$, has normal vectors
\[
(\nu'_1,0), \ldots, (\nu'_{d'},0),
(0,\nu''_1), \ldots, (0,\nu''_{d''}) \in \Z^{n'} \times \Z^{n''}
= \Z^n\,.
\]
Hence, the linear map $\pP$ on $CF^n = CF^{n'}\otimes CF^{n''}$ is 
given by
\[
\pP (\eps'\otimes\eps'') = (\pPp\eps')\otimes\eps'' + 
\eps'\otimes(\pPd\eps'')\,.
\]
The result of the proposition follows by a standard application
of the K\"unneth formula in this context.
\end{proof}

Using this proposition and the fact that $P \times P$ always has an even number of facets, 
we can define the Floer invariant of any integral polytope.
\begin{defn}\label{defn:HF}
The Floer invariant $HF(P)$ of an integral polytope $P$ is defined by
\[
HF(P) := \sqrt{\dim\,\ker(\pPP) - \dim\,\im(\pPP)}\,.
\]
\end{defn}

\subsection{Relation with Lagrangian Floer Homology}
\label{sec:LFH}

Suppose now that $P$ is a Fano Delzant polytope and let $M_P$ denote its associated smooth Fano 
toric variety. This means that when defining $P\subset(\R^n)^\ast$ by
\[
\ell_i (x) := \langle x, \nu_i \rangle + a_i \geq 0\,,\ 
i=1,\ldots,d.
\]
we can assume that $a_1 = \cdots = a_d = 1$ and $(M_P,\omega_P)$ is monotone: $[\omega_P] = 2\pi c_1 (M_P)$.

Denote by $R_P$ the Lagrangian real part of $(M_P,\omega_P)$  and by $T_P$ the Lagrangian torus over
$0\in\interior(P)$. Assuming that the Lagrangian Floer homology $HF(R_P, T_P; \Z_2)$ is well defined, we
have the following theorem. 

\begin{thm}\label{thm:main} 
\[
\dim HF(R_P, T_P; \Z_2) = HF (P).
\]
\end{thm}

\begin{cor} \label{cor:main}
\[
\sharp (\psi(R_P)\pitchfork T_P) \geq HF(P)\,,\ \forall\, \psi\in\Ham(M_P,\omega_P)\,.
\]
\end{cor}
\begin{proof}[Proof of the Corollary]
If $P$ is a Fano Delzant polytope, then $P\times P$ is an even Fano
Delzant polytope to which Theorem~\ref{thm:main} applies. Suppose that
\[
\sharp (\psi(R_P)\cap T_P) < HF(P)\,.
\]
Then $(\psi\times\psi)\in\Ham(M_P\times M_P, \omega_P \times \omega_P)$
would be such that
\[
\sharp\left[(\psi\times\psi)(R_P\times R_P) \,\cap\, (T_P \times T_P) \right]
= \left(\sharp (\psi(R_P) \,\cap\, T_P) \right)^2 < 
(HF(P))^2 = HF (P\times P)\,,
\]
which would imply that
\[
\dim HF(R_{P\times P}, T_{P\times P}; \Z_2 ) < HF (P\times P)\,.
\]
This contradicts Theorem~\ref{thm:main}. As before, we are assuming that
$HF(R_{P\times P}, T_{P\times P}; \Z_2 )$ is well defined.
\end{proof}

\begin{remark}
The meaning of $HF(R_P, T_P; \Z_2)$ being well defined depends on the technical tools one is
willing to use. If one uses only a basic version of Lagrangian Floer homology, then the condition that the minimal Maslov number of $R_P$ is bigger than two has to be enforced. Under this assumption, it can be proved as in \cite{Alston,Cho1,Cho-Oh,Oh} that, for the standard complex structure $J$ on $M_P$, the boundary of the moduli space of holomorphic strips with dimension one is given by broken strips and holomorphic disks with Maslov index two and boundary in $T_P$. Moreover, the linearized Cauchy-Riemann operator $D_u\bar\partial_{J}$ is surjective for all $J$-holomorphic strips $u$ of Maslov index one or two. It follows from this that $HF(R_P, T_P; \Z_2)$ is well defined whenever $d$ is even. As a matter of fact, if $d$ is even then the holomorphic disks with index two and boundary in $T_P$ counted by $\partial_P^2$ cancel out when we work with $\Z_2$ coefficients; this is the key point in the proof of Proposition \ref{prop:square}. In particular, $HF(R_{P\times P}, T_{P\times P}; \Z_2 )$ is well defined whatever is the parity of $d$. Indeed, the minimal Maslov number of $R_{P\times P}$ equals the minimal Maslov number of $R_P$.

With this hypothesis on $R_P$ the set of examples that are covered by Corollary~\ref{cor:main} is a bit restrictive. If one uses
a sophisticated version of Lagrangian Floer homology, such as the one developed by Fukaya, Oh, 
Ohta and Ono in~\cite{FOOO1, FOOO2, FOOO}, then Corollary~\ref{cor:main} covers a lot more
ground (see~\cite{Alston-Amorim}).
\end{remark}

\begin{remark} \label{rmk:even}
Even with only basic Lagrangian Floer homology, $HF(R_{P\times P}, T_{P\times P}; \Z_2 )$ is indeed 
well defined when $P$ is a simplex, i.e. $M_P = \CP^n$, and Corollary~\ref{cor:main} can be combined 
with a combinatorial computation of $HF(P\times P)$ to prove Theorem~\ref{thm:even}, hence removing 
the dimension restriction of Alston's result. We will not present that combinatorial computation here since
(i) it is presented in~\cite{Alston-Amorim} and (ii), as we will see, our remark on symplectic reduction 
can also be used to easily remove this restriction.
\end{remark}

The proof of Theorem~\ref{thm:main} is a simple combination of the 
following two ingredients:
\begin{itemize}
\item[(i)] the characterization by Cho~\cite{Cho1} and Cho-Oh~\cite{Cho-Oh}
of the holomorphic discs on Fano toric manifolds, that are relevant for
the differential on the Lagrangiam Floer complex, as Blaschke products;
\item[(ii)] the existence of \emph{global homogeneous coordinates} on
any smooth toric variety $M_P$, not just on $\C P^n$ (see \S 4.4 of 
Cox~\cite{Cox}).
\end{itemize}
We will now discuss some details of (ii), and how they contribute to the
proof of Theorem~\ref{thm:main}.

From the well known quotient representation
\[
M_P = \left(\C^d \setminus Z_P \right) / G_P
\]
we get a map
\begin{align}
\C^d \setminus Z_P & \longrightarrow M_P \notag \\
(z_1,\ldots,z_d) & \longmapsto [z_1,\ldots,z_d] \notag
\end{align}
defining homogeneous coordinates on $M_P$, where homogeneous here
is with respect to the action of the torus $G_P$. Note that the canonical
anti-holomorphic involution $\tau_P : M_P \to M_P$ is given in these
homogeneous coordinates as
\begin{equation}\label{eq:tau}
\tau_P ([z_1,\ldots,z_d]) = [\oz_1, \ldots, \oz_d]\,.
\end{equation}

Let $m_1,\ldots,m_\ell$ be the integral points in $P$, i.e.
\[
\left\{ m_1, \ldots, m_\ell\right\} = P \cap \Z^n\,,
\]
and consider the map
\[
\phi : \C^d \setminus Z_P \to \C P^{\ell-1}
\]
given by
\begin{equation}\label{eq:phi}
\phi (z_1,\ldots,z_d) = [z^{[m_1]}, \ldots,z^{[m_\ell]}]\,, 
\end{equation}
where
\[
z^{[m]} = \prod_{i=1}^d z_i^{\ell_i (m)}\,.
\]
It turns out that $\phi$ induces a well defined embedding
\begin{align}
M_P & \longrightarrow \C P^{\ell -1} \notag \\
[z_1,\ldots,z_d] & \longmapsto [z^{[m_1]},\ldots,z^{[m_\ell]}] \,. \notag
\end{align}
If one restricts this map to $(\C^\ast)^n \subset M_P$, one gets
an embedding
\[
\varphi: (\C^\ast)^n \to \C P^{\ell-1}
\]
given in the standard coordinates of $(\C^\ast)^n$ as
\begin{equation}\label{eq:varphi}
\varphi (t_1, \ldots, t_n) = [t^{m_1}, \ldots, t^{m_\ell}]
\end{equation}
where 
\[
t^{m_k} = \prod_{i=1}^n t_i^{(m_k)_i} \,.
\]

Since
\[
R_P \cap (\C^\ast)^n = (\R^\ast)^n \quad\text{and}\quad
T_P = (S^1)^n \subset (\C^\ast)^n
\]
we get that
\[
R_P \cap T_P = \left\{ (\eps_1,\ldots,\eps_n): \eps_i = \pm 1 \right\} \,.
\]
Any such intersection point can be written in homogeneous coordinates
as
\[
[\eps_1,\ldots, \eps_d]\,,\ \eps_i = \pm 1\,.
\]
As Alston~\cite{Alston} does in the case of $\C P^n$, using the 
work of Cho~\cite{Cho1} and Cho-Oh~\cite{Cho-Oh}, the characterization 
of relevant holomorphic discs as Blaschke products shows that the
Lagrangian Floer differential is given in homogeneous coordinates
by
\[
\pP ([\eps_1,\ldots, \eps_d]) = 
\sum_{i=1}^d [\eps_1,\ldots, -\eps_i, \ldots, \eps_d] \,.
\]

We can now use the maps $\phi$ and $\varphi$, defined by~(\ref{eq:phi})
and~(\ref{eq:varphi}), to compute
$\pP (\eps_1, \ldots, \eps_n)$ directly. Note that if
\[
\phi([\eps_1,\ldots, \eps_i, \ldots, \eps_d]) =
[\eps^{[m_1]}, \ldots, \eps^{[m_\ell]}]
\]
then
\begin{align}
\phi([\eps_1,\ldots, -\eps_i, \ldots, \eps_d]) 
& = [(-1)^{\ell_i(m_1)}\eps^{[m_1]}, \ldots, 
(-1)^{\ell_i(m_\ell)}\eps^{[m_\ell]}] \notag \\
& = [(-1)^{\langle m_1,\nu_i\rangle+a_i}\eps^{[m_1]}, \ldots, 
(-1)^{\langle m_\ell,\nu_i\rangle+a_i}\eps^{[m_\ell]}] \notag \\
& = [(-1)^{\langle m_1,\nu_i\rangle}\eps^{[m_1]}, \ldots, 
(-1)^{\langle m_\ell,\nu_i\rangle}\eps^{[m_\ell]}] \,.\notag 
\end{align}
On the other hand, if
\[
\varphi (t_1, \ldots, t_n) = [t^{m_1}, \ldots, t^{m_\ell}] =
[\eps^{[m_1]}, \ldots, \eps^{[m_\ell]}]
\]
then
\begin{align}
\varphi ((-1)^{(\nu_i)_1} t_1, \ldots, (-1)^{(\nu_i)_n} t_n)
& = [(-1)^{\langle m_1,\nu_i\rangle} t^{m_1}, \ldots,
(-1)^{\langle m_\ell,\nu_i\rangle} t^{m_\ell}] \notag \\
& = [(-1)^{\langle m_1,\nu_i\rangle} \eps^{[m_1]}, \ldots,
(-1)^{\langle m_\ell,\nu_i\rangle} \eps^{[m_\ell]}] \notag \\
& = \phi([\eps_1,\ldots, -\eps_i, \ldots, \eps_d]) \,.\notag
\end{align}
Hence,
\begin{align}
& \pP ([\eps_1,\ldots, \eps_d]) = 
\sum_{i=1}^d [\eps_1,\ldots, -\eps_i, \ldots, \eps_d] \notag \\
\Leftrightarrow \qquad &
\pP (\eps_1,\ldots, \eps_n) = 
\sum_{i=1}^d \left((-1)^{(\nu_i)_1}\eps_1,\ldots, 
(-1)^{(\nu_i)_n}\eps_n\right) \,.\notag
\end{align}

\subsection{Examples}
\label{sec:EX}

\begin{example}
Let $P=[-1,1]\subset\R$ with $M_P = \C P^1 = S^2$. In this case 
$R_P$ is a meridian circle and $T_P$ is the equator. $P$ has normals
$\nu_1 = 1$ and $\nu_2 = -1$. $CF$ is a $2$-dimensional vector space
with basis $e_1 = (1)$ and $e_2 = (-1)$. The differential $\pP$ is given 
by
\[
\pP (\eps) = ((-1)^1 \eps) + ((-1)^{-1} \eps) = 2 (-\eps) = 0 \,.
\]
Hence,
\[
HF(P) = 2\,.
\]
\end{example}

\begin{prop}\label{prop:even-nu}
If P is an even polytope such that whenever $\nu$ is a normal $-\nu$ is also
a normal, i.e. such that $P$ is \emph{symmetric}, then $HF(P) = 2^n$.
\end{prop}
\begin{proof}
As in the example above, we always have $\pP = 0$ under these conditions.
\end{proof}

\begin{remark} Under the assumption that $HF(R_{P}, T_{P}; \Z_2 )$ is well defined,
it follows from this proposition that if $P$ is symmetric then
\[
\sharp (\psi (T_P) \cap R_P) \geq 2^n \quad\text{for any $\psi\in\Ham (M_P, \omega_P)$.}
\]
In fact, as we will see in Proposition~\ref{prop:sym}, one can easily use our remark on symplectic 
reduction to remove the assumption and give another proof of this result.
\end{remark}

\begin{example}
When $n=2$ there are exactly two even Fano polytopes with the property of
Proposition~\ref{prop:even-nu}: the Fano square, corresponding to 
$\C P^1 \times \C P^1$, and the Fano hexagon, corresponding to $\C P^2$
blown up at $3$ points.
\end{example}

\begin{example}
Let $P\subset \R^2$ be the Fano simplex corresponding to $M_P = \C P^2$.
Since $P$ has $3$ facets, an odd number, we will consider the even Fano 
Delzant polytope $P\times P\subset\R^4$, whose $6$ facets have normals
\[
\nu_1 = (1,0,0,0)\,,\ \nu_2 = (0,1,0,0)\,,\ \nu_3 = (-1,-1,0,0)\,,\ 
\nu_4 = (0,0,1,0)\,,\ \nu_5 = (0,0,0,1)
\]
\[
\text{and}\quad
\nu_6 = (0,0,-1,-1)\,. 
\]
This means that $\pPP:CF^4 \to CF^4$ is given by
\begin{align}
\pPP (\eps_1,\eps_2,\eps_3,\eps_4) & = (-\eps_1,\eps_2,\eps_3,\eps_4) + 
(\eps_1,-\eps_2,\eps_3,\eps_4) + (-\eps_1,-\eps_2,\eps_3,\eps_4)
\notag \\
& + (\eps_1,\eps_2,-\eps_3,\eps_4) + (\eps_1,\eps_2,\eps_3,-\eps_4) +
(\eps_1,\eps_2,-\eps_3,-\eps_4)\,. \notag
\end{align}
With this explicit formula it is not hard to check that
\[
HF (P\times P) = \dim\,\ker(\pPP) - \dim\,\im(\pPP) = 10 - 6 = 4\,.
\]
Hence, 
\[
HF(P) = \sqrt{4} = 2
\]
and applying Corollary~\ref{cor:main} we conclude that
\[
\sharp (\psi(\R P^2)\cap \T^2) \geq 2\,,
\]
for any $\psi\in\Ham(\C P^2)$. This estimate is known to be optimal (see for example the end of
section $5$ in~\cite{Alston-Amorim}).
\end{example}

\section{Second Remark: Symplectic Reduction} 
\label{sec:rmk-2}

Here we will state some elementary general facts in the particular context of symplectic
toric manifolds. 

Let $(\tM,\tomega)$ be a symplectic toric manifold of dimension $2N$ with
$\tilde\T$-action generated by a moment map 
\[
\tmu:\tM\to\tP\subset (\R^N)^\ast\,.
\]
As before, given $x\in\interior (\tP)$, let $\tT_x := \tmu^{-1} (x)$ denote the corresponding $\tilde\T$-orbit,
a Lagrangian torus in $\tM$, and let $\tR$ denote the real part of $\tM$, i.e. the Lagrangian submanifold
given by the fixed point set of the canonical anti-symplectic involution $\ttau:\tM\to\tM$, characterized by
$\tmu \circ \ttau = \tmu$. Recall that
\[
\ttau (\ttheta\cdot\tp) = -\ttheta \cdot (\ttau(\tp))\,,\ \forall\ \ttheta\in\tilde\T\,, \ \tp\in\tM\,,
\quad\text{and}\quad 
\ttau (\tT_x) = \tT_x\,.
\]
Moreover, $(\tT_x)^{\ttau} = \tT_x \cap \tR$ and this set, the real part of a regular
$\tilde\T$-orbit, is discrete with $2^N$ points.

Let $K\subset\tilde\T$ be a subtorus of dimension $N-n$ determined by an inclusion of Lie algebras
\[
\iota:\R^{N-n} \to \R^N\,.
\]
The moment map for the induced action of $K$ on $\tM$ is
\[
\tmu_K = \iota^\ast \circ\tmu : \tM \to (\R^{N-n})^\ast\,.
\]
Let $c\in\tmu_K (\tM) \subset  (\R^{N-n})^\ast$ be a regular value and assume that $K$ acts freely on the 
level set
\[
Z := \tmu_K^{-1} (c) \subset \tM\,.
\]
Then, the reduced space $(M := Z/K, \omega)$ is a symplectic toric manifold of dimension $2n$ with
$\T:=\tilde\T /K$-action generated by a moment map
\[
\mu : M \to P \subset (\R^n)^\ast \cong \ker (\iota^\ast)
\]
that fits in the commutative diagram
\begin{equation*}  
\begin{CD}
\tM \supset Z @>{\tmu}>> \tP\subset (\R^N)^\ast \\
@V{\pi}VV            @AA{}A\\
M @>{\mu}>>  P \subset (\R^n)^\ast 
\end{CD}
\end{equation*}
where $\pi$ is the quotient projection and the vertical arrow on the right is the inclusion 
$(\R^n)^\ast \cong \ker (\iota^\ast) \subset (\R^{N-n})^\ast$.

Recall that the reduced symplectic form $\omega$ is characterized by $\pi^\ast \omega = \tomega |_Z$.
Note that given $T_x := \mu^{-1}(x)$, with $x\in \interior (P)\subset\interior(\tP)$, we have that
\[
\pi^{-1} (T_x) = \tT_x\,.
\]
Moreover, $\ttau (Z) = Z$, $Z^{\ttau} = Z \cap \tR$ and $\ttau$ induces the canonical anti-symplectic
involution $\tau:M\to M$ via
\[
\pi \circ \ttau = \tau \circ \pi\,.
\]
Let $p\in R:= M^\tau$. Then a simple counting argument shows that
\[
\sharp (\pi^{-1} (p) \cap \tR) = 2^{N-n}\,.
\]

\begin{lemma} \label{lem:lift}
Let $\psi\in\Ham (M,\omega)$. Then there is $\tpsi\in\Ham(\tM, \tomega)$ such that
$\tpsi (Z) = Z$ and
\[
\pi (\tpsi (\tp)) = \psi (\pi (\tp))\,,\ \forall\, \tp\in Z\,.
\]
\end{lemma}
\begin{proof}
Given a time dependent hamiltonian $h_t :M\to\R$ let ${\tilde h}_t :\tM\to\R$ be any smooth
extension to $\tM$ of $h_t \circ \pi : Z \to \R$. The Hamiltonian flow generated by ${\tilde h}_t$
has the desired properties.
\end{proof}

\begin{prop} \label{prop:displace} Let $\psi\in\Ham(M,\omega)$, $\tpsi\in\Ham(\tM,\tomega)$
a lift given by the previous lemma and $x\in\interior (P) \subset \interior (\tP)$. Then
\begin{itemize}
\item[(i)] $\sharp (\psi (T_x) \cap R) = r \Rightarrow \sharp (\tpsi (\tT_x) \cap \tR) = r 2^{N-n}$.
\item[(ii)] $\psi (T_x) \cap T_x = \emptyset \Rightarrow \tpsi (\tT_x) \cap \tT_x = \emptyset$.
\end{itemize}
Moreover
\begin{itemize}
\item[(iii)] $\sharp (\psi (T_x) \pitchfork T_x) = r \Rightarrow \exists\  \tvphi\in\Ham(\tM,\tomega)\ \text{such that}\ \,
 \sharp (\tvphi (\tT_x) \pitchfork \tT_x) = r 2^{N-n}$.
\end{itemize}
\end{prop}
\begin{remark}
McDuff's method of probes~\cite{McDuff} can be seen as a particular case of (ii).
\end{remark}
\begin{proof}
To prove (i), suppose that $\tq\in\tpsi(\tT_x)\cap\tR$ and let $\tp = \tpsi^{-1} (\tq) \in \tT_x$.
Then $p=\pi(\tp)\in T_x$ and
\[
\psi(p) = \psi(\pi(\tp)) = \pi (\tpsi(\tp)) = \pi(\tq) = \pi (\ttau (\tq)) = \tau (\pi(\tq)) = \tau(q) = \tau (\psi(p))\,.
\]
Hence, $q=\psi (p) \in R$ and so $q\in \psi (T_x) \cap R$.

Moreover, given $q\in \psi (T_x) \cap R$ let $\tq \in \pi^{-1} (q) \cap \tR$. Let $p = \psi^{-1} (q) \in T_x$
and $\tp = \tpsi^{-1} (\tq)$. Then
\begin{align}
& \pi (\tp) = \pi (\tpsi^{-1} (\tq)) = \psi^{-1} (\pi(\tq)) = \psi^{-1} (q) = p \notag \\
\Rightarrow\  & \tp \in \pi^{-1} (p) \subset \pi^{-1} (T_x) = \tT_x \notag \\
\Rightarrow\  & \tq = \tpsi (\tp) \in \tpsi (\tT_x) \,.\notag
\end{align}
Hence, $\tq\in\tpsi(\tT_x) \cap \tR$ and there are exactly $2^{N-n}$ such $\tq$'s.

To prove (ii), suppose that $\tq\in\tpsi(\tT_x)\cap\tT_x$ and let $\tp = \tpsi^{-1} (\tq) \in \tT_x$.
Then $p=\pi(\tp), q = \pi(\tq) \in T_x$ and 
\[
\psi (p) = \psi (\pi (\tp)) = \pi (\tpsi(\tp)) = \pi (\tq) = q
\]
which implies that $q\in \psi (T_x) \cap T_x$.

To prove (iii), note that since $T_x$ and $\psi (T_x)$ intersect transversely at $r$ points in $M$,
say $p_1,\ldots,p_r\in M$, we have that $\tT_x$ and $\tpsi (\tT_x)$ intersect in $Z$ in a Morse-Bott way
along the $r$ orbits of the subtorus $K\subset\tilde\T$ given by $\pi^{-1}(p_1), \ldots,\pi^{-1}(p_r)\subset \tM$.
Standard equivariant neighborhood theorems in symplectic geometry imply that a sufficiently small 
neighborhood $\tilde\Uu\subset\tM$ of each of these isotropic tori is $K$-equivariantly symplectomorphic
to $\Vv_1 \times \Vv_2 \subset (\R^{2n}, \omega_{\rm st}) \times (T^\ast K, \omega_{\rm can})$, where
$\Vv_1\subset\R^{2n}$ is a neighborhood of the origin, $\Vv_2 \subset T^\ast K$ is a neighborhood of the 
$0$-section, and
\[
\omega_{\rm st} = du \wedge dv = \sum_{j=1}^n du_j \wedge dv_j \quad\text{and} \quad
\omega_{\rm can} = - d\lambda_{\rm can}
\]
are the usual symplectic forms on $\R^{2n}$ and $T^\ast K$ respectively. Moreover, we can identify $\Vv_1$
with a neighborhood $\Uu \subset M$ of the point in $\psi (T_x) \pitchfork T_x$ under consideration and require that
\[
T_x \cap \Uu \cong \left\{ (u,v)\in\Vv_1\,:\ v=0\right\}\quad\text{and}\quad
\psi(T_x) \cap \Uu \cong \left\{ (u,v)\in\Vv_1\,:\ u=0\right\}\,.
\]
Let $\varphi\in\Ham^c (\Vv_2, \omega_{\rm can})$ be such that $\sharp (\varphi (\text{$0$-section})  
\pitchfork (\text{$0$-section}) ) = 2^{N-n}$ (one can clearly construct such optimal displacing Hamiltonians
supported in arbitrarily small neighborhoods of the $0$-section in $T^\ast K$). We can then consider
${\rm id} \times \varphi : \Vv_1 \times \Vv_2 \to \Vv_1 \times \Vv_2$, extend as the identity to $\tM$ and 
compose with $\tpsi$ to obtain a Hamiltonian that perturbs the relevant intersection $K$-orbit into
$2^{N-n}$ transversal intersection points. By doing that at each of the $r$ points in $\psi (T_x) \pitchfork T_x$ 
we obtain the desired $\tvphi\in\Ham(\tM,\tomega)$.
\end{proof}

\begin{cor} \label{cor:displace} \ 
\begin{itemize}
\item[(i)] If $\sharp (\tpsi (\tT_x) \cap \tR) \geq m$ for any $\tpsi\in\Ham (\tM, \tomega)$ then
\[
\sharp (\psi (T_x) \cap R) \geq \frac{m}{2^{N-n}} \,,\ \forall\, \psi\in\Ham (M, \omega)\,.
\]
\item[(ii)] If $\tT_x \subset (\tM, \tomega)$ is non-displaceable, then $T_x \subset (M, \omega)$ is also
non-displaceable.
\item[(iii)]  If $\sharp (\tpsi (\tT_x) \pitchfork \tT_x) \geq m$ for any $\tpsi\in\Ham (\tM, \tomega)$ then
\[
\sharp (\psi (T_x) \pitchfork T_x) \geq \frac{m}{2^{N-n}} \,,\ \forall\, \psi\in\Ham (M, \omega)\,.
\]

\end{itemize}
\end{cor}
\begin{remark}
This idea of using symplectic reduction to prove intersection properties of Lagrangian submanifolds
was used by Tamarkin in~\cite{Tamarkin}. It is also present in the work of Borman~\cite{Bo1} on reduction
properties of quasi-morphisms and quasi-states (see also~\cite{Bo2}).
\end{remark}

\subsection{Symplectic Reduction Construction of Toric Manifolds}
\label{ssec:reduction}

Recall that any symplectic toric manifold $(M^{2n},\omega)$ can be constructed as a symplectic reduction
of 
\[
(\tM = \R^{2d}\cong \C^{2d},\tomega = \sum_{j=1}^d dx_j \wedge dy_j)\,, 
\]
where $d$ is the number of facets of the corresponding polytope $P\subset (\R^n)^\ast$. This reduction is 
with respect to the natural action of a subtorus $K\subset \tilde{\T} = \T^d$ of dimension $d-n$, whose Lie algebra 
$\Lie (K) \subset\R^d = \Lie (\T^d)$ is determined as the kernel of the linear map
\begin{equation} \label{eq:beta}
\beta:\R^d \to \R^n\,, \quad \beta (e_j) = \nu_j\,,\ j=1,\ldots,d\,,
\end{equation}
where $\{e_1,\ldots,e_d\}$ is the canonical basis of $\R^d$ and $\nu_1,\ldots,\nu_d \in \Z^n \subset \R^n$ are
the primitive integral interior normals to the facets of the moment polytope $P$. 

When $K = K_1 \times K_2 \subset \T^d$, correspondng to a splitting $\Lie (K) = \Lie (K_1) \times \Lie (K_2) \subset \R^d$,
recall that the principle of reduction in stages tells us that, at appropriate level sets, reduction with respect to the action
of $K\subset\T^d$ is equivalent to 
\begin{itemize}
\item[-] first reducing with respect to $K_1\subset\T^d$, obtaining a symplectic manifold $(M_1, \omega_1)$ with 
Hamiltonian action of $\T^d / K_1$,
\item[-] then reducing $(M_1, \omega_1)$ with respect to $K_2 \subset \T^d / K_1$.
\end{itemize}
This principle will be used repeatedly in the applications considered in the next sections. It is also
the main ingredient in the proof of the following proposition, which in turn will be used in the proof
of Theorem~\ref{thm:monotone} (cf. Proposition~\ref{prop:monotone} 
in \S~\ref{ssec:monotone}).

\begin{prop} \label{prop:factor}
Let $(M^{2n},\omega)$ be a symplectic toric manifold and  $\nu_1, \ldots, \nu_d \in \Z^n$
the primitive integral interior normals to the facets of its moment polytope $P\subset (\R^n)^\ast$. 
Let $m_1, \ldots, m_d \in\N$ be such that
\begin{equation} \label{eq:weighted}
\sum_{j=1}^d m_j \nu_j = 0\,.
\end{equation}
Then $(M^{2n}, \omega)$ can be obtained as a symplectic reduction of the weighted projective space
$\CP (m_1, \ldots, m_d)$.
\end{prop}

\begin{proof}
$(M^{2n},\omega)$ can be obtained as the symplectic reduction of $(\R^{2d}, dx \wedge dy)$
by the action of a subtorus $K\subset\T^d$ with $\Lie(K) = \ker\beta$, where the linear map
$\beta$ is given by~(\ref{eq:beta}). This means in particular that
\[
(k_1, \ldots, k_d) \in \Lie (K) \subset \R^d \Leftrightarrow \sum_{j=1}^d k_j \nu_j = 0 \,,
\]
which together with~(\ref{eq:weighted}) implies that $K$ can be written as $K = K_1 \times K_2$ with
$\Lie (K_1) = \text{span} \{(m_1,\ldots,m_d)\}$. Since the weighted projective space $\CP (m_1, \ldots, m_d)$
is obtained as the symplectic reduction of $(\R^{2d}, dx \wedge dy)$ by the action of $K_1$, one
can use the principle of reduction in stages to conclude the proof.
\end{proof}

\section{First Application: Monotone Cases}
\label{sec:monotone}

We will use the following results stated in Theorem~\ref{thm:basic}:
\begin{itemize}
\item[(i)] The Clifford torus $T^n \subset \CP^n$ is non-displaceable and 
$\sharp (\psi (T_0) \pitchfork T_0) \geq 2^n\,,\ \forall\, \psi \in \Ham (\CP^n,\omega_{FS})$~\cite{BEP,Cho1}.
\item[(ii)] The pair $(\RP^{2n-1}, T^{2n-1})$ is non-displaceable in $\CP^{2n-1}$ and
$\sharp (\psi (T^{2n-1}) \pitchfork \RP^{2n-1}) \geq 2^n$ for any $\psi\in\Ham (\CP^{2n-1})$~\cite{Alston}.
\end{itemize}

\begin{defn} \label{defn:centred}
Recall from Example~\ref{ex:monotone} that any monotone toric symplectic manifold has a unique monotone
torus fiber, called the centered or special torus fiber, which is the Lagrangian torus orbit over the ``center'' of its
moment polytope. A symplectic reduction of a monotone toric symplectic manifold at a level containing its centered
torus fiber, i.e. through the special ``center'' of its moment polytope, will be called a \emph{centered} symplectic 
reduction.
\end{defn}
 
\subsection{Application 1}  \label{ssec:even}
One can easily extend Alston's result to $\CP^{2n}$. In fact, $\CP^{2n}$ can
be obtained as a centered symplectic reduction of $\CP^{2n+1}$ (cf. Figure~\ref{fig:1} for the $n=1$ case) and 
we get that
\[
\sharp (\psi (T^{2n}) \pitchfork \RP^{2n}) \geq \frac{2^{n+1}}{2} = 2^n \quad\text{for any $\psi\in\Ham (\CP^{2n})$.}
\]
Note that the fact that this estimate is known to be optimal for $\CP^2$ implies that Alston's bound for
$\CP^3$ is also optimal.

\begin{figure}[ht]
      \centering
      \includegraphics{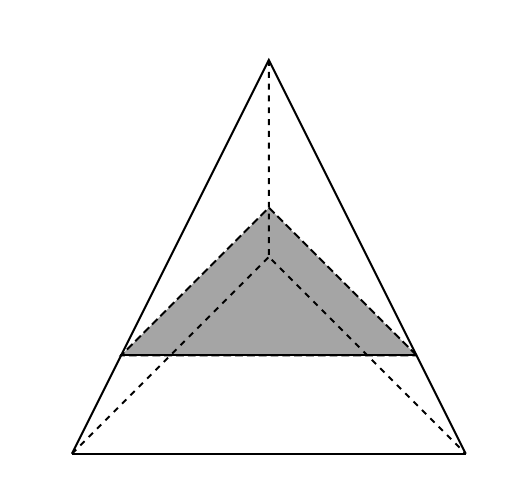}
      \caption{$\CP^2$ as reduction of $\CP^3$.}
      \label{fig:1}
  \end{figure}

\subsection{Application 2} \label{ssec:zero-sum}
Let $(M^{2n},\omega)$ be a monotone toric manifold, $R$ its real part and $T$ 
its special centered torus fiber. Let $\nu_1, \ldots, \nu_d \in \Z^n$ denote the primitive integral interior normals 
to the facets of the moment polytope of $M$.

\begin{prop} \label{prop:zero-sum}
If $\sum_{i=1}^d \nu_i = 0$ then $T$ and the pair $(R,T)$ are non-displaceable. Moreover,
\[
\sharp (\psi (T) \pitchfork T) \geq 2^n \,, \ \forall\, \psi \in \Ham (M,\omega)\,.
\]
\end{prop}
\begin{proof}
Proposition~\ref{prop:factor} and the zero-sum condition on the normals imply that the standard symplectic 
reduction construction of $M$ from $\C^d$ factors through $\CP^{d-1}$. The monotone condition implies that 
this is a centered symplectic reduction and we can apply  Corollary~\ref{cor:displace}.
\end{proof}

\subsection{Application 3} \label{ssec:sym}
The monotone $M = \CP^2 \sharp 3 \overline{\CP}^2$, i.e. the monotone blow-up
of $\CP^2$ at three points, fits the context of Application 2 and can be obtained as a centered symplectic
reduction of $\CP^5$. This means that
\[
\sharp (\psi (T) \pitchfork R) \geq \frac{2^3}{2^3} = 1\quad\text{for any $\psi\in\Ham (M)$.}
\]
It can also be obtained as a centered symplectic reduction of $\CP^2 \times \CP^2$. Since
our Floer combinatorial invariant of $\CP^2 \times \CP^2$ is $4$, this gives the same bound:
\[
\sharp (\psi (T) \pitchfork R) \geq \frac{4}{2^2} = 1\quad\text{for any $\psi\in\Ham (M)$.}
\]
However, if one sees $M$ as a centered symplectic reduction of $\CP^1 \times \CP^1 \times \CP^1$, 
cf. Figure~\ref{fig:2}, one improves the bound to
\[
\sharp (\psi (T) \pitchfork R) \geq \frac{2^3}{2} = 4\quad\text{for any $\psi\in\Ham (M)$,}
\]
which is optimal and coincides with the value of the Floer combinatorial invariant of the hexagon.

\begin{figure}[ht]
      \centering
      \includegraphics{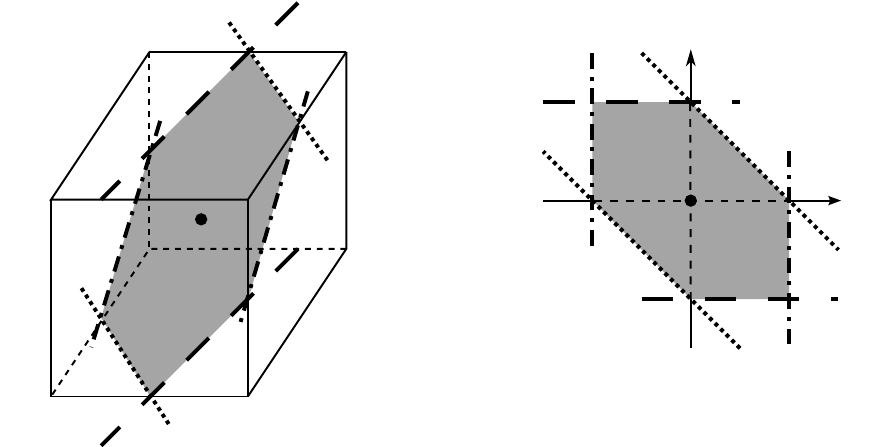}
      \caption{$\CP^2 \sharp 3 \overline{\CP}^2$ as reduction of $\CP^1 \times \CP^1 \times \CP^1$.}
      \label{fig:2}
  \end{figure}

Let us describe the details of this reduction construction. The cube at the left of Figure~\ref{fig:2}, corresponding
to the moment polytope of a monotone $\CP^1 \times \CP^1 \times \CP^1$, can be described by the following
inequalites: 
\begin{align}
x_1 + 1 \geq 0 \qquad & \qquad\quad\qquad\ x_2 + 1 \geq 0 & x_3 + 1 \geq 0 \notag \\
-x_1 + 1 \geq 0 \qquad & \qquad\qquad -x_2 + 1 \geq 0 &  -x_3 + 1 \geq 0 \notag
\end{align}
The monotone $\CP^2 \sharp 3 \overline{\CP}^2$ can be obtained from this toric manifold by centered symplectic
reduction with respect to the circle $S^1 \subset \T^3$ determined by the Lie algebra vector 
$(-1,-1,1)\in\R^3 = \Lie (\T^3)$, at level $-x_1 -x_2 + x_3 = 0$. The quotient $2$-torus $\T^3 / S^1$ acts on the 
reduced manifold and, with respect to its Lie algebra basis given by 
$(1,0,0), (0,1,0) \in \Lie (\T^3 / S^1) \cong \R^3 / \{(-1,-1,1)\}$, the resulting moment polytope is described by the 
above inequalities with $x_3 = x_1 + x_2$, i.e. the hexagon at the right of Figure~\ref{fig:2}.

In fact, the monotone $\CP^2 \sharp 3 \overline{\CP}^2$ is just a particular case of the following
more general proposition.

\begin{prop} \label{prop:sym}
Let $(M^{2n},\omega)$ be a monotone toric manifold, $R$ its real part and $T$ its special centered torus 
fiber. Suppose that the corresponding moment polytope $P\subset\R^n$ is \emph{symmetric}, i.e.
if $\nu\in\Z^n$ is the interior normal to a facet of $P$ then $-\nu$ is also the interior normal
to a facet of $P$. Then
\[
\sharp (\psi (T) \pitchfork R) \geq 2^n \quad\text{for any $\psi\in\Ham (M)$ and this bound is optimal.}
\]
\end{prop}
\begin{proof}
The fact that the polytope $P$ is symmetric implies that $M$ can be obtained as a symplectic reduction
of the product of $d$ copies of $\CP^1$, where $2d$ is the number of facets of $P$. The fact that $M$
is monotone implies that all the $\CP^1$'s have the same area and that this is a centered symplectic 
reduction. Hence, we get from Corolary~\ref{cor:displace} that
\[
\sharp (\psi (T) \pitchfork R) \geq \frac{2^d}{2^{d-n}}  = 2^n\,,\ \forall\, \psi\in\Ham (M, \omega)\,.
\]
Since $\sharp (T \pitchfork R) = 2^n$ the bound is indeed optimal.
\end{proof}

\subsection{Application 4} \label{ssec:monotone}

Let $(M^{2n},\omega)$ be a compact monotone symplectic toric manifold and $T$ its special centered torus fiber.
Denote by $\nu_1,\ldots\nu_d\in\Z^n$ the primitive integral interior normals to the facets of its Delzant polytope
$P\subset(\R^n)^\ast$.

\begin{lemma} \label{lem:monotone}
There exists $k\in\{1,\ldots,d\}$ such that
\[
\nu_k + \sum_{j=1, j\ne k}^d m_j \nu_j = 0
\]
with all $m_j \in \N$.
\end{lemma}

\begin{proof}
Since $P$ is the moment polytope of a compact toric manifold, the support of its associated fan is the whole $\R^n$.
In particular, every lattice vector $\nu\in\Z^n$ belongs to a cone of the fan determined by some vertex of $P$, which
means that it can be written as a non-negative integral linear combination of the primitive integral interior normals to 
the $n$ facets that meet at that vertex. The Lemma follows by taking
\[
\nu = - \sum_{j=1}^d \nu_j\,.
\]
\end{proof}

\begin{prop} \label{prop:monotone}
On any compact monotone toric symplectic manifold $(M, \omega)$ the special centered Lagrangian
torus fiber $T$ is non-displaceable. Moreover,
\[
\sharp (\psi (T) \pitchfork T) \geq 2^n \,, \ \forall\, \psi \in \Ham (M,\omega)\,.
\]
\end{prop}

\begin{proof}
Using the previous Lemma, and after a possible re-ordering of the normals, we can assume that
\[
\sum_{j=1}^{d-1} m_j \nu_j + \nu_d = 0 \quad\text{with $m_1, \ldots, m_{d-1} \in\N$.}
\]
This condition and Proposition~\ref{prop:factor} imply that the standard symplectic reduction construction 
of $M$ from $\C^d$ factors through the weighted projective space $\CP(1, m_1, \ldots, m_{d-1})$. The monotone 
condition implies that this factorization goes through the special centered non-displaceable torus fiber of this 
weighted projective space (cf. Theorem~\ref{thm:weighted}) and we can apply  Corollary~\ref{cor:displace}.
\end{proof}

As a particular example, consider the monotone $M = \CP^2 \sharp \overline{\CP}^2$, i.e. the monotone
blow-up of $\CP^2$ at one point, with polytope $P\subset (\R^2)^\ast$ given by
\begin{align}
x_1 + 1 \geq 0 \quad (\nu_1 = (1,0)) & \ & -x_1 -x_2 + 1 \geq 0 \quad (\nu_3 = (-1,-1))\notag \\
x_2 + 1 \geq 0 \quad (\nu_2 = (0,1)) & \ & x_1  +x_2 + 1 \geq 0 \quad (\nu_4 = (1,1)).\quad\ \notag
\end{align}
We have that
\[
(\nu_1 + \nu_2 + 2\nu_3) + \nu_4 = 0
\]
and $M$ can be obtained as a ``centered" symplectic reduction of $\CP (1,1,1,2)$. In fact, the polytope
of $\CP (1,1,1,2)$ is given by
\begin{align}
x_1 + 1 \geq 0 & \ & x_3 + 1 \geq 0 \notag \\
x_2 + 1 \geq 0 & \ & -x_1  -x_2 -2x_3 + 1 \geq 0 \notag
\end{align}
and its reduction with respect to the circle $S^1\subset\T^3$ determined by $(1,1,1)\in\R^3 = \Lie (\T^3)$
at level 
\[
x_1 + x_2 + x_3 = 0\,, \ \text{i.e.}\   x_3 = -x_1 -x_2\,,
\]
gives rise to the above polytope for $M = \CP^2 \sharp \overline{\CP}^2$.

\section{Second Application: non-monotone Fano cases }
\label{sec:Fano}

For the applications in this section we will assume that some form of the following general result is true:
\begin{itemize}
\item If $T_i$ or the pair $(T_i, R_i)$ have $HF\neq 0$ or are non-displaceable in $(M_i, \omega_i)$,
$i=1,2$, then the same is true for the corresponding $T_1 \times T_2$ and 
$(T_1 \times T_2, R_1 \times R_2)$ in $(M_1 \times M_2, \omega_1 \times \omega_2)$.
\end{itemize}

\begin{remark}
For Lagrangian torus orbits in toric symplectic manifolds, the set-up of Woodward~\cite{Woodward} applies
and proves a result of this form (cf.~\cite{WW}).
\end{remark}

Moreover, in some of the applications below we will also use the following result:
\begin{itemize}
\item In the total space of the line bundle $\Oo(-1)\to\CP^1$, the special torus $T$ sitting over the
origin in the polygon on the right side of Figure~\ref{fig:3} is non-displaceable. This has been proved 
by Woodward (cf. Example~1.3 in~\cite{Woodward}) and can also be seen as a consequence of a result 
of Cho in~\cite{Cho2} (cf. polygon on the left side of Figure~\ref{fig:3}).
\end{itemize}

\begin{figure}[ht]
      \centering
      \includegraphics{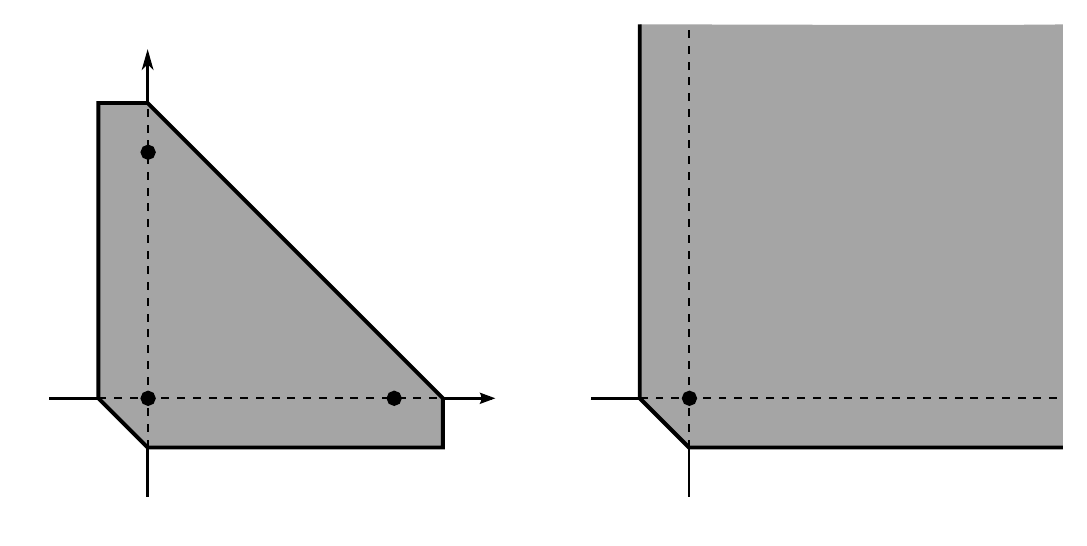}
      \caption{Cho's result on the left and its limit on the right, giving a non-displaceable torus fiber in the 
      total space of the line bundle $\Oo(-1)\to\CP^1$.}
      \label{fig:3}
  \end{figure}

\subsection{Application 5} Consider $M = \CP^2 \sharp \overline{\CP}^2$, i.e. the blow-up of $\CP^2$ at 
one point. 

Figure~\ref{fig:4} illustrates how one can obtain two non-displaceable torus fibers when the exceptional
divisor is small, i.e. smaller than monotone. On the left, one thinks of $M$ as a symplectic reduction 
of $\CP^2 \times \CP^1$, with $\CP^2$ ``smaller" than $\CP^1$, to show that the torus fiber over the origin
is non-displaceable. On the right, one thinks of $M$ as a symplectic reduction of  $\Oo(-1)\times\CP^2$
to show that a torus fiber ``close" to the exceptional divisor is non-displaceable.

\begin{figure}[ht]
      \centering
      \includegraphics[scale=0.9]{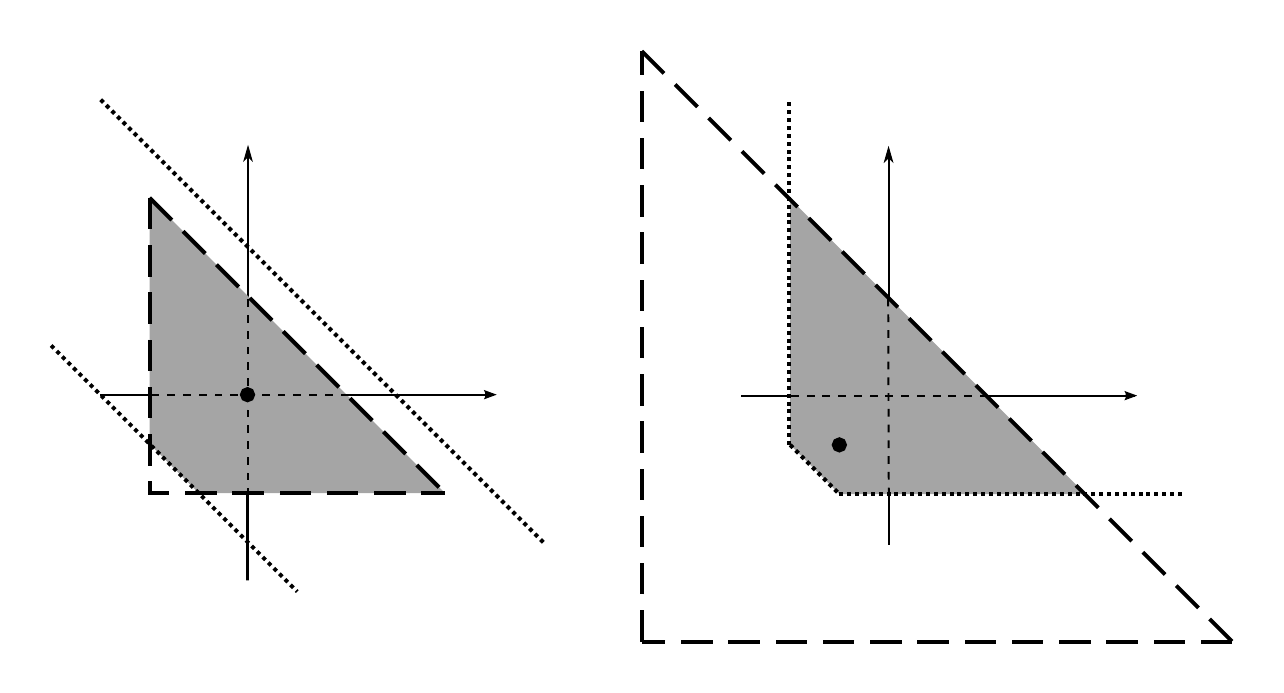}
      \caption{$\CP^2 \sharp \overline{\CP}^2$ as reduction of $\CP^2 \times \CP^1$ (on the left, $\CP^2$ 
      ``smaller" than $\CP^1$) and of $\Oo(-1)\times\CP^2$ (on the right).}
      \label{fig:4}
  \end{figure}

Figure~\ref{fig:5} illustrates how one can obtain one non-displaceable torus fiber when the exceptional
divisor is big, i.e. bigger than monotone. One thinks again of $M$ as a symplectic reduction of 
$\CP^2 \times \CP^1$, but now with $\CP^2$ ``bigger" than $\CP^1$.

\begin{figure}[ht]
      \centering
      \includegraphics{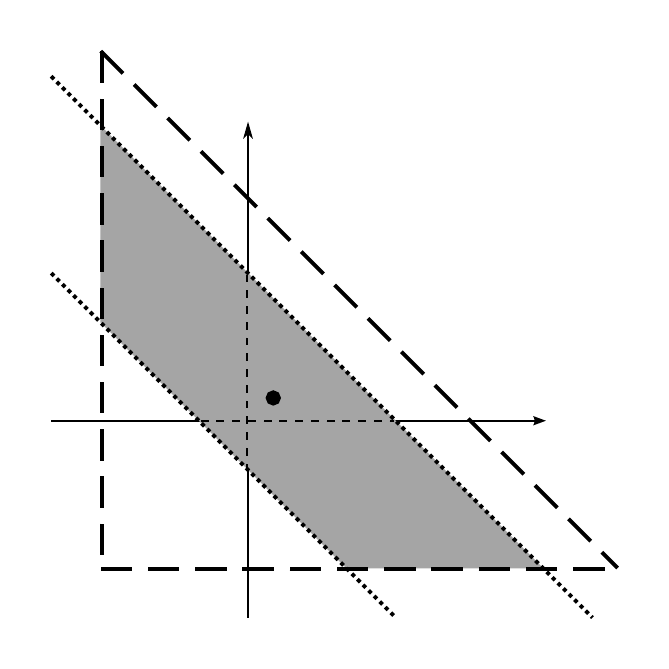}
      \caption{$\CP^2 \sharp \overline{\CP}^2$ as reduction of $\CP^2 \times \CP^1$, now with $\CP^2$ 
      ``bigger" than $\CP^1$.}
      \label{fig:5}
  \end{figure}

Note that the monotone case with just one non-displaceable torus fiber over the special central point
can be obtained as a limit of any of these.

\subsection{Application 6} A very similar idea applies to $M = \CP^2 \sharp 2 \overline{\CP}^2 
\cong (\CP^1 \times \CP^1) \sharp \overline{\CP}^2$, i.e. the equal blow-up of $\CP^2$ at two points which can
also be thought of as the blow-up of $\CP^1 \times \CP^1$ at one point. One recovers the results of 
Fukaya-Oh-Ohta-Ono~\cite{FOOO1, FOOO2} and Woodward~\cite{Woodward} illustrated in 
Figures~\ref{fig:6} and~\ref{fig:7}.

Figure~\ref{fig:6} illustrates how one can obtain two non-displaceable torus fibers in the big two-point blow-up
of $\CP^2$, which is equivalent to a small one-point blow-up of $\CP^1 \times \CP^1$.
On the left, one thinks of $M$ as a symplectic reduction of $\CP^2 \times \CP^1\times \CP^1$, with ``big" $\CP^2$,
to show that the torus fiber over the origin is non-displaceable. On the right, one thinks of $M$ as a symplectic 
reduction of  $\Oo(-1)\times\CP^2$ to show that a torus fiber ``close" to the blown-up point on $\CP^1 \times \CP^1$
is non-displaceable.

\begin{figure}[ht]
      \centering
      \includegraphics[scale=0.75]{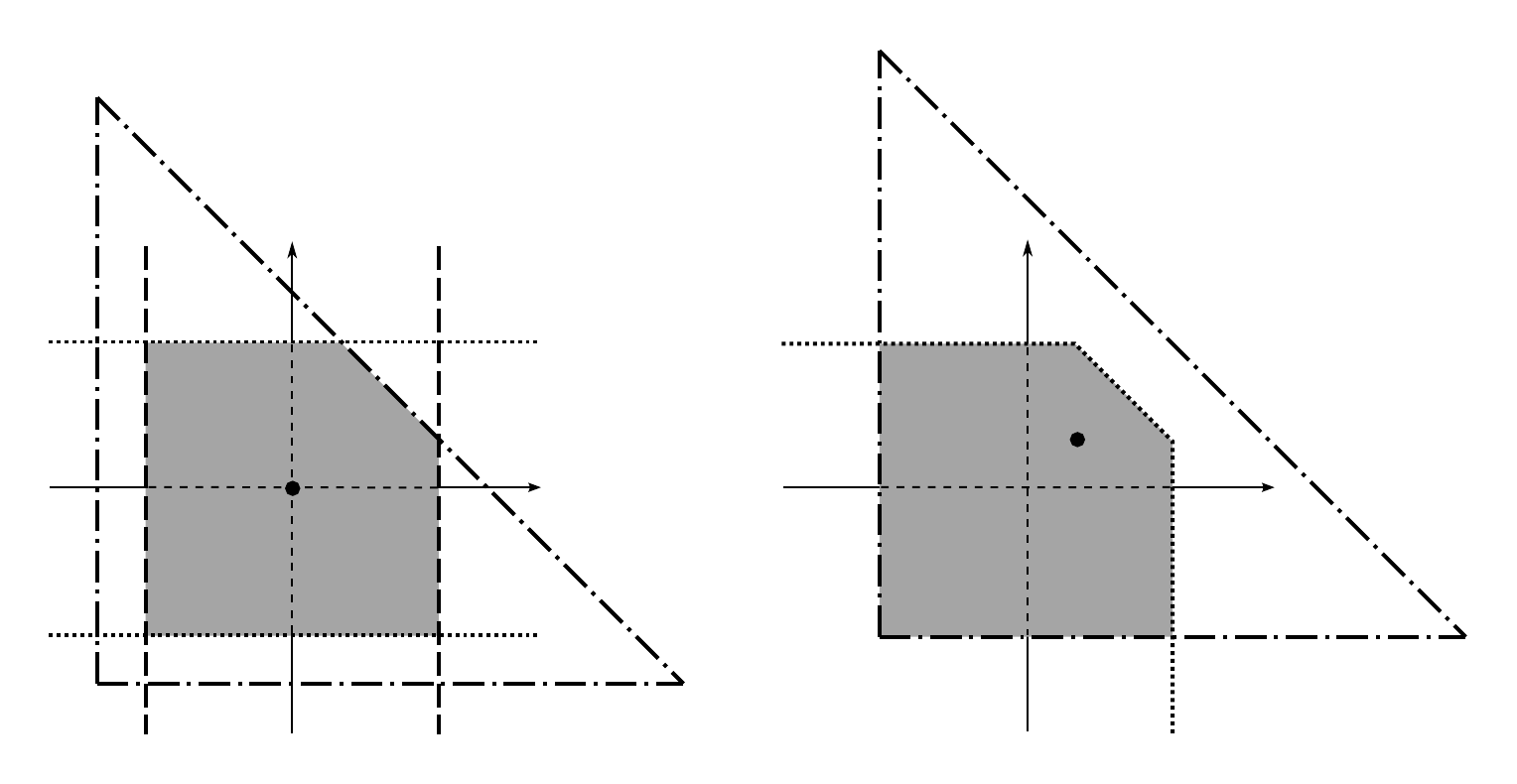}
      \caption{$(\CP^1 \times \CP^1) \sharp \overline{\CP}^2$ (small blow-up), equivalently 
      $\CP^2 \sharp 2 \overline{\CP}^2$ (big blow-ups), as reduction of $\CP^2 \times \CP^1\times \CP^1$ 
      (on the left) and of $\Oo(-1)\times\CP^2$ (on the right).}
      \label{fig:6}
  \end{figure}

Figure~\ref{fig:7} illustrates how one can obtain three non-displaceable torus fibers in the small two-point blow-up
of $\CP^2$, which is equivalent to a big one-point blow-up of $\CP^1 \times \CP^1$.
On the left, one thinks of $M$ as a symplectic reduction of $\CP^2 \times \CP^1\times \CP^1$, with ``small" $\CP^2$
and large $\CP^1$'s, to show that a torus fiber ``close" to the origin is non-displaceable. On the right, one thinks of 
$M$ as a symplectic reduction of  $\Oo(-1)\times\CP^1\times\CP^1$ to show that there is a non-displaceable torus 
fiber ``close" to each blown-up point on $\CP^2$.

\begin{figure}[ht]
      \centering
      \includegraphics[scale=0.75]{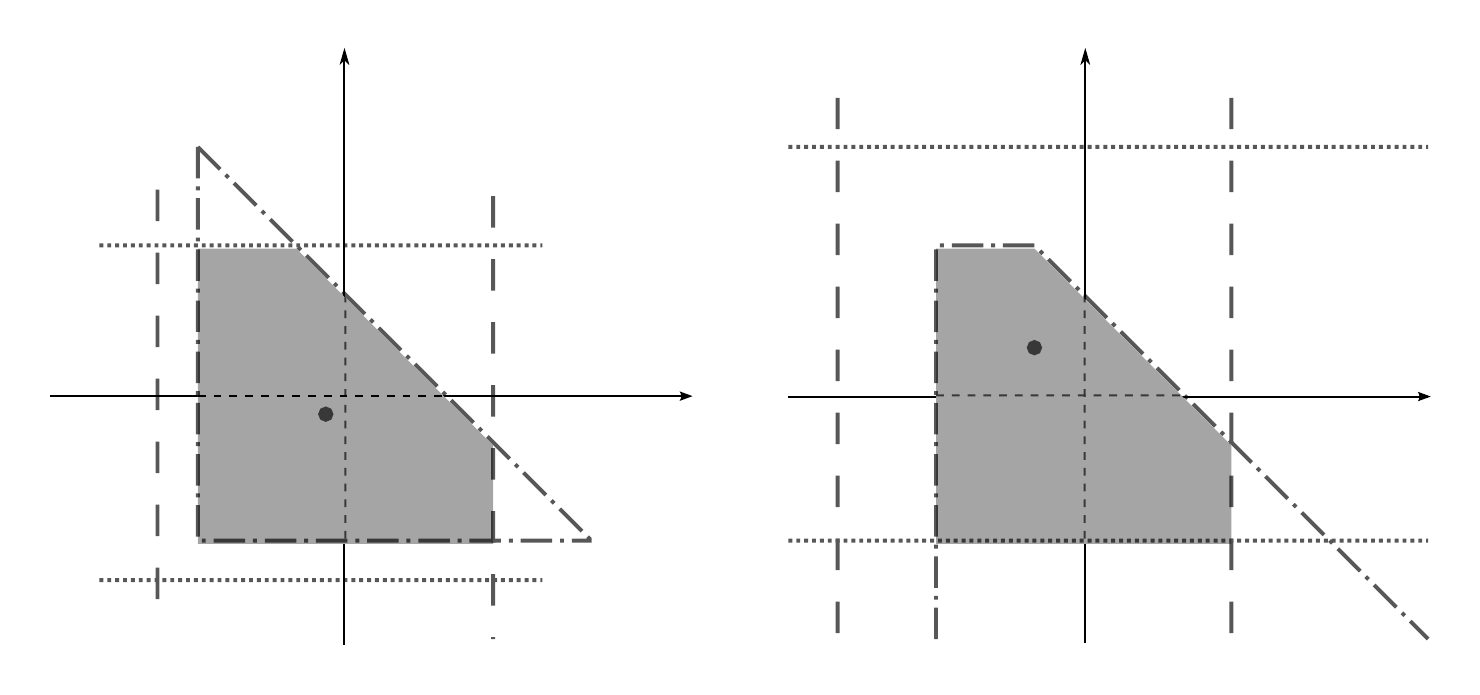}
      \caption{$(\CP^1 \times \CP^1) \sharp \overline{\CP}^2$ (big blow-up), equivalently 
      $\CP^2 \sharp 2 \overline{\CP}^2$ (small blow-ups), as reduction of $\CP^2 \times \CP^1\times \CP^1$ 
      (on the left) and of $\Oo(-1)\times\CP^1 \times \CP^1$ (on the right).}
      \label{fig:7}
  \end{figure}

Again, note that the monotone case with just one non-displaceable torus fiber over the special central point 
can be obtained as a limit of any of these. It could also be obtained using Proposition~\ref{prop:monotone}, as
was illustrated in~\S~\ref{ssec:monotone} for the monotone one-point blow-up $\CP^2 \sharp \overline{\CP}^2$.

\subsection{Application 7} \label{ssec:cont}
Here we will use the same idea to understand an example Fukaya, Oh, Ohta and Ono~\cite{FOOO1, FOOO2}, presented 
in Figure~\ref{fig:8} (see Example 10.3 in~\cite{FOOO}). The symplectic manifold is $M = \CP^2 \sharp 2 \overline{\CP}^2$ 
with blow-ups of different sizes, one smaller than monotone and the other bigger than monotone, and they obtain a closed 
interval of non-displaceable torus fibers. This can also be obtained by considering $M$ as the symplectic reduction of 
$\Oo(-1)\times \CP^1 \times \CP^1$ (or the compact $(\CP^2 \sharp \overline{\CP}^2)\times\CP^1 \times \CP^1$) as shown 
in Figure~\ref{fig:8}.

\begin{figure}[ht]
      \centering
      \includegraphics[scale=0.95]{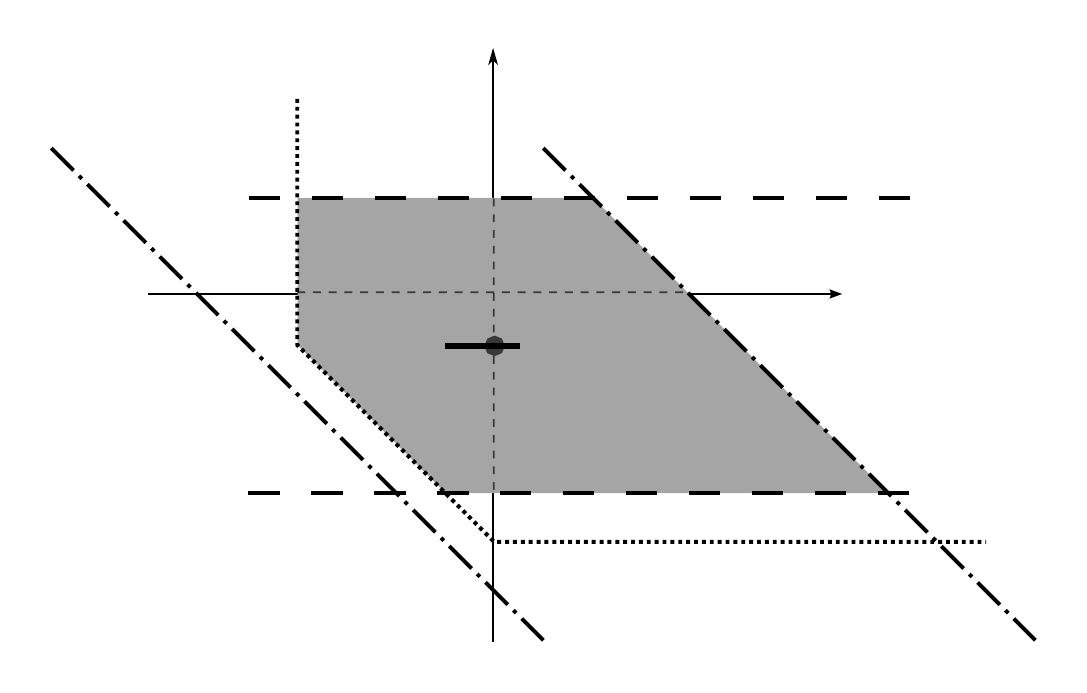}
      \caption{$\CP^2 \sharp 2 \overline{\CP}^2$ as reduction of $\Oo(-1)\times \CP^1 \times \CP^1$.}
      \label{fig:8}
  \end{figure}

The details are as follows. Consider $(\CP^2 \sharp 2 \overline{\CP}^2, \omega_\alpha)$ given by the
Delzant polytope $P_\alpha \subset (\R^2)^\ast$ determined by the following inequalities:
\begin{align}
x_1 + 1 \geq 0 & \ & x_1 + x_2 + 1 + \alpha \geq 0 \notag \\
x_2 + 1 \geq 0 & \ & -x_2 + (1-2\alpha) \geq 0 \notag \\
-(x_1 + x_2) +1 \geq 0 & \ & \text{with $0< \alpha < 1$.}\notag
\end{align}
The non-displaceable torus fibers are over the points with coordinates $(-\alpha + \lambda, -\alpha)$ for
$0 < \lambda < 3\alpha / 2$ (Figure~\ref{fig:8} corresponds to $\alpha = \lambda = 1/4$). To prove that for
each such pair of real numbers $\alpha$ and $\lambda$, we consider $\Oo(-1)\times \CP^1 \times \CP^1$ 
with moment polytope given by the cartesian product of the following polytopes:
\begin{itemize}
\item[-] the one for the $\Oo(-1)$ factor is given by the inequalities
\[
x_1 + 1 \geq 0 \,,\quad x_1 + x_2 + 1 + \alpha \geq 0 \quad\text{and}\quad x_2 + 1 + \lambda \geq 0\,,
\]
having a non-displaceable torus fiber over the point with coordinates $x_1 = -\alpha + \lambda$ and $x_2 = -\alpha$.
\item[-] the one for the first $\CP^1 $ factor is given by the inequalities
\[
x_3 + 1 \geq 0 \quad\text{and}\quad -x_3 + (1-2\alpha) \geq 0\,,
\]
having a non-displaceable torus fiber over the point with coordinate $x_3 = -\alpha$.
\item[-] the one for the second $\CP^1 $ factor is given by the inequalities
\[
x_4 + 1 + 4\alpha -2\lambda \geq 0 \quad\text{and}\quad -x_4 + 1 \geq 0\,,
\]
having a non-displaceable torus fiber over the point with coordinate $x_4 = -2\alpha+\lambda$.
\end{itemize}
We can now do symplectic reduction with respect to the $2$-torus $\T^2 \subset \T^4$ determined
by the Lie algebra vectors $(0,-1,1,0), (-1,-1,0,1)\in\R^4 = \Lie (\T^4)$ at the level given by
\[
x_3 = x_2 \quad\text{and}\quad x_4 = x_1 + x_2\,,
\]
to obtain the polytope $P_\alpha$ with non-displaceable torus fiber over the point with coordinates
$x_ 1 = -\alpha + \lambda$ and $x_2 =-\alpha$.

\subsection{Application 8} One can use the same idea to understand non-displaceable torus fibers
on $M = \CP^2 \sharp 3 \overline{\CP}^2$ for all possible sizes of blown-up points. Figure~\ref{fig:9} 
illustrates the case of three small size blow-ups, where one gets four non-displaceable torus fibers. 

\begin{figure}[ht]
      \centering
      \includegraphics[scale=0.9]{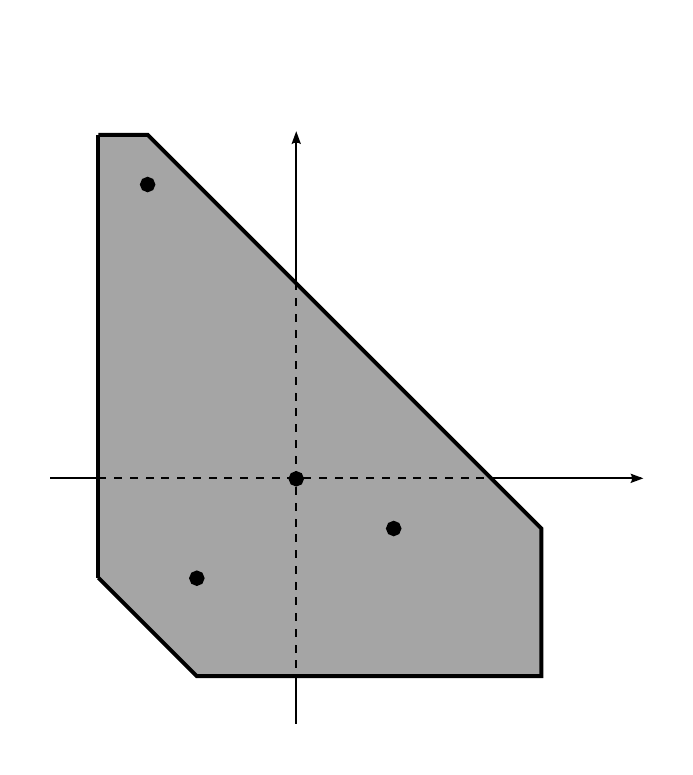}
      \caption{Four non-displaceable torus fibers on $\CP^2 \sharp 3 \overline{\CP}^2$.}
      \label{fig:9}
  \end{figure}

The center fiber is obtained by seeing $M$ as the symplectic reduction of $\CP^2\times\CP^1\times\CP^1\times\CP^1$, 
as in the left side of Figure~\ref{fig:10}, while each of the off-center fibers is obtained by seeing $M$ as the symplectic 
reduction of $\Oo(-1)\times\CP^1\times\CP^1\times\CP^1$, as in the right side of Figure~\ref{fig:10}.

\begin{figure}[ht]
      \centering
      \includegraphics[scale=0.65]{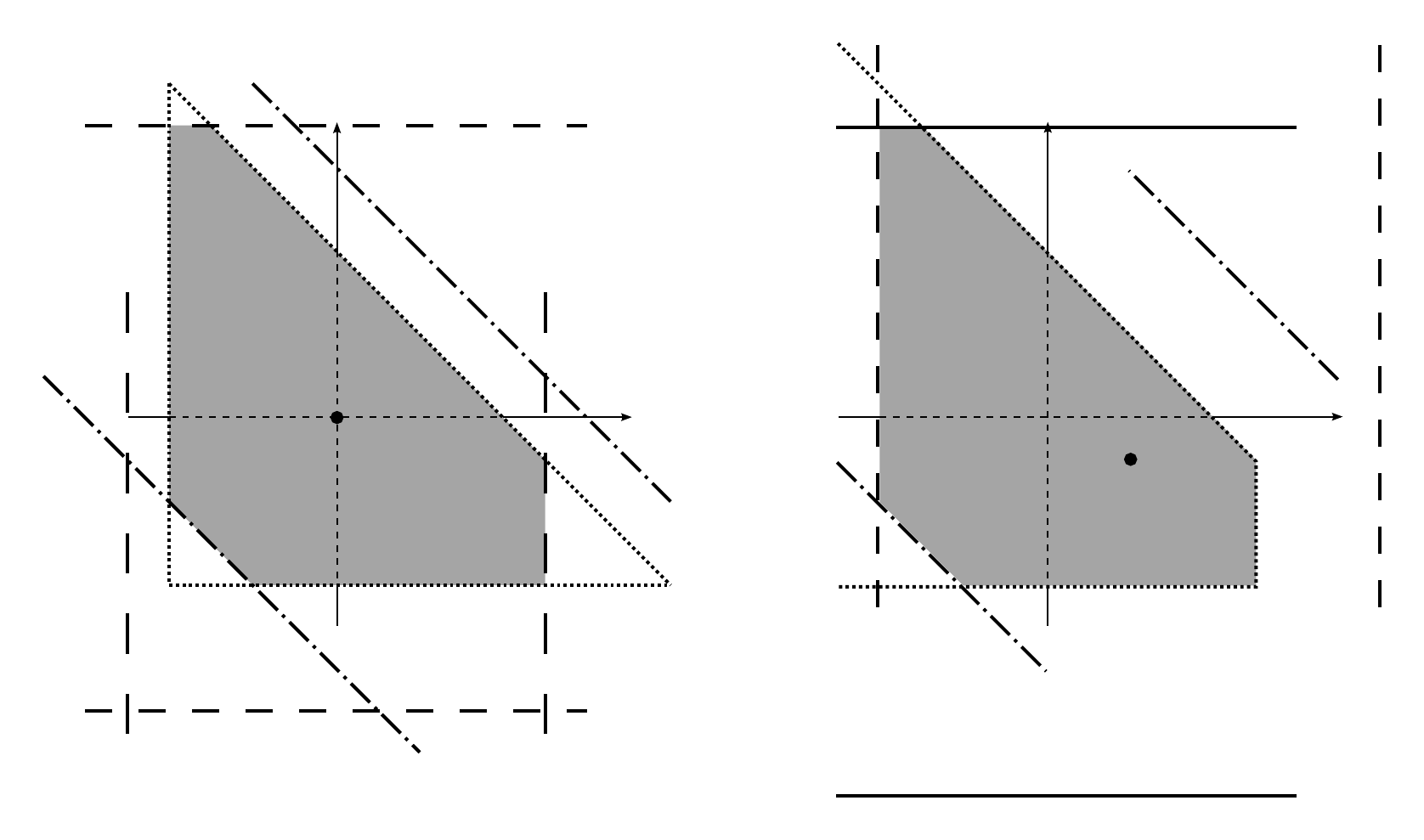}
      \caption{$\CP^2 \sharp 3 \overline{\CP}^2$ as reduction of $\CP^2\times\CP^1\times\CP^1\times\CP^1$ 
      (on the left) and of $\Oo(-1)\times\CP^1\times\CP^1\times\CP^1$ (on the right).}
      \label{fig:10}
  \end{figure}

Note that we can also get a closed interval of non-displaceable torus fibers for 
$M = \CP^2 \sharp 3 \overline{\CP}^2$ , e.g.  by blowing up the example in Application 6 at the lower right 
corner of Figure~\ref{fig:8}.

\section{Third Application: non-Fano cases}
\label{sec:non-Fano}

Here we will use the basic non-displaceability result for the central torus fiber of a weighted projective
space stated in Theorem~\ref{thm:weighted} and show how it implies non-displaceability
of at least one torus fiber on any Hirzebruch surface, a result of Fukaya-Oh-Ohta-Ono~\cite{FOOO1} (see 
Example 10.1 in~\cite{FOOO}), and a continuous interval of non-displaceable torus fibers on a particular
non-Fano toric surface considered by McDuff~\cite{McDuff2}.

\subsection{Application 9} \label{ssec:Hirz}

Consider the weighted projective space $\CP (1,1,k)$, the symplectic quotient of $\C^3$
by the $S^1$-action with weights $(1,1,k)$, with moment polytope $P_k \subset \R^2$ given by
\[
x_1 + 1 \geq 0\,,\quad x_2 + 1 \geq 0 \quad\text{and}\quad -x_1 -k x_2 +1 \geq 0\,.
\]
The special torus sitting over the origin is non-displaceable. Figure~\ref{fig:11} illustrates the $k=2$ case.

\begin{figure}[ht]
      \centering
      \includegraphics[scale=1.0]{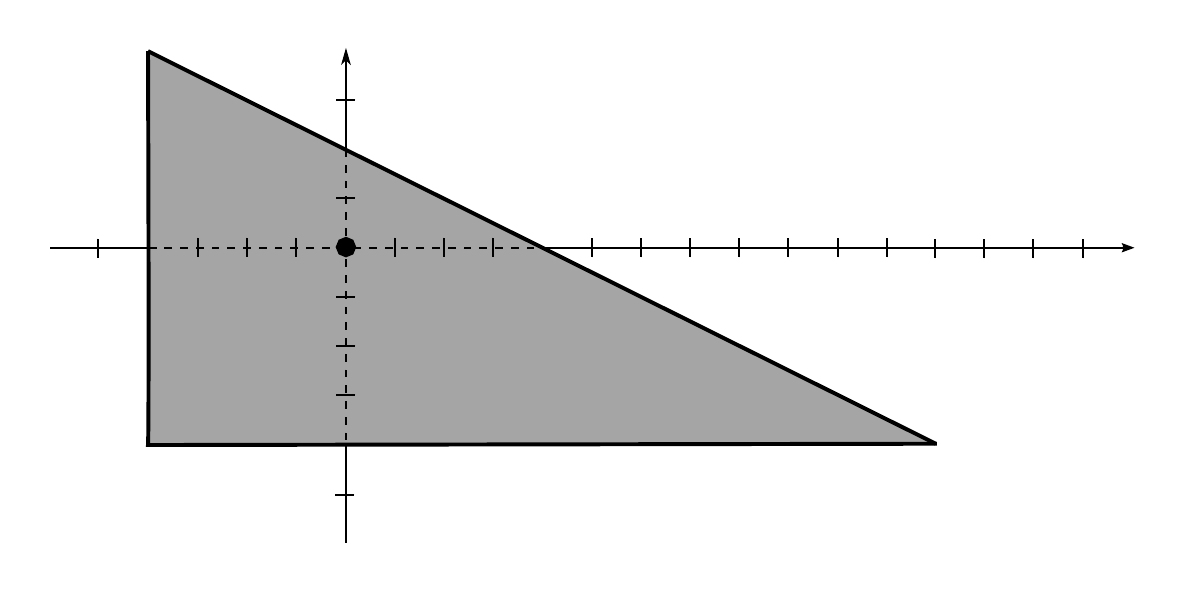}
      \caption{$\CP (1,1,2)$.}
      \label{fig:11}
  \end{figure}

Now let $H_k := \Proj (\Oo (-k) \oplus \C) \to \CP^1$ be a Hirzebruch surface, with $2\leq k \in\N$. 
Each of these Hirzebruch surfaces can be seen as a centered symplectic reduction of $\CP (1,1,k) \times \CP^1$ 
and that implies at least one non-displaceable torus fiber on any $H_k$. 
Figure~\ref{fig:12} illustrates the $k=2$ case.

\begin{figure}[ht]
      \centering
      \includegraphics[scale=1.0]{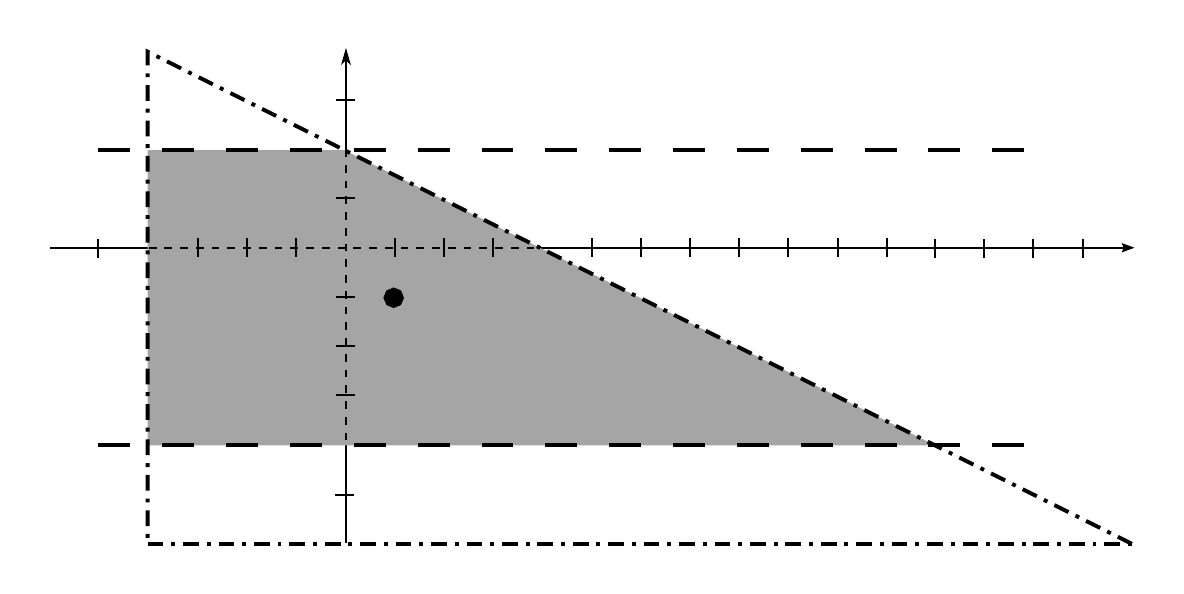}
      \caption{The Hirzebruch surface $H_2$ as reduction of $\CP (1,1,2) \times \CP^1$.}
      \label{fig:12}
  \end{figure}

\subsection{Application 10} \label{ssec:Dusa}

Here we will consider the non-Fano toric symplectic $4$-manifold described by McDuff in~\S~2.1 of~\cite{McDuff2}.
We already pointed out its potential significance in Remark~\ref{rmk:Dusa}.

Consider the Delzant polytope $P\subset (\R^2)^\ast$ determined by the following inequalities:
\begin{align}
x_1 +1 \geq 0 & \ & -x_1 - 3 x_2 + 3 \geq 0 \notag \\
x_2 +1 \geq 0 & \ & -x_1 - 2x_2 +3 \geq 0 \notag \\
-x_2 +1 \geq 0 & \ & \ \notag
\end{align}
The torus fibers over the points with coordinates $x_1 = \lambda$ and $x_2 = 0$, with $1 < \lambda < 2$, 
are non-displaceable by probes. We can show that all these fibers are in fact non-displaceable by
considering this toric $4$-manifold as a symplectic reduction of $\CP (1,1,2) \times \CP^1 \times \Oo(-1)$,
as shown in Figure~\ref{fig:13} (where $\lambda = 3/2$).

\begin{figure}[ht]
      \centering
      \includegraphics[scale=1.0]{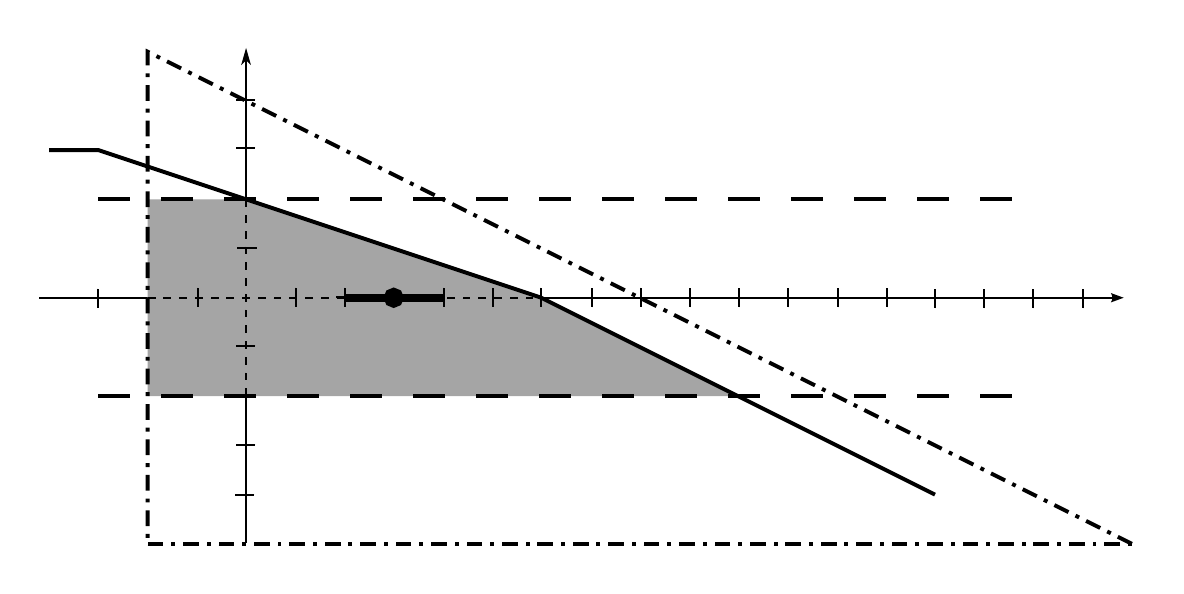}
      \caption{Non-Fano toric surface with interval on non-displaceable torus fibers.}
      \label{fig:13}
  \end{figure}

More precisely, for each $1 < \lambda < 2$, consider $\CP (1,1,2) \times \CP^1 \times \Oo(-1)$
with moment polytope given by the cartesian product of the following polytopes:
\begin{itemize}

\item[-] the one for the $\CP (1,1,2)$ factor is given by the inequalities
\[
x_1 + 1 \geq 0 \,,\quad x_2 + 1 + \lambda \geq 0 \quad\text{and}\quad -x_1 -2 x_2 + 1 + 2\lambda \geq 0\,,
\]
having a non-displaceable torus fiber over the point with coordinates $x_1 = \lambda$ and $x_2 = 0$.

\item[-] the one for the $\CP^1 $ factor is given by the inequalities
\[
x_3 + 1 \geq 0 \quad\text{and}\quad -x_3 + 1 \geq 0\,,
\]
having a non-displaceable torus fiber over the point with coordinate $x_3 = 0$.

\item[-] the one for the $\Oo(-1)$ factor is given by the inequalities
\[
x_4 + 3 \geq 0 \,,\quad x_5 + 3 - \lambda \geq 0 \quad\text{and}\quad x_4 + x_5 + 3 \geq 0\,,
\]
having a non-displaceable torus fiber over the point with coordinate $x_4 = -\lambda$ and $x_5 = 0$.
\end{itemize}
We can now do symplectic reduction with respect to the $3$-torus $\T^3 \subset \T^5$ determined
by the Lie algebra vectors $(0,-1,1,0,0), (-1,-2,0,1,0), (0, -1,0,0,1)\in\R^5 = \Lie (\T^5)$ at the level given by
\[
x_3 = x_2 \,,\quad x_4 = -x_1 -2x_2 \quad\text{and}\quad x_5 = - x_2\,,
\]
to obtain the polytope $P$ with non-displaceable torus fiber over the point with coordinates
$x_ 1 = \lambda$ and $x_2 = 0$.

%\newpage

\end{document}